\documentclass[lettersize,journal]{IEEEtran}
\usepackage{orcidlink}
\usepackage{graphicx}
\usepackage{amsfonts}
\usepackage{amsthm}
\usepackage{amsmath, amssymb}
\usepackage{hyperref}
\usepackage{cleveref}
\usepackage[normalem]{ulem}
\usepackage{algorithm}
\usepackage[noend]{algorithmic}
\usepackage{stmaryrd} 
\newtheorem{thm}{Theorem}
\newtheorem{lem}{Lemma}

\newtheorem{prop}[lem]{Proposition}

\theoremstyle{remark}
\newtheorem{defn}{Definition}

\newtheorem{remark}{Remark}
\hypersetup{colorlinks=true}

\usepackage[square,numbers]{natbib}
\newcommand{\crv}{\cite{Cardona09}}

\usepackage{xcolor}
\definecolor{changecolor}{RGB}{192,64,0}

\newcommand{\com}[1]{}

\usepackage{enumitem}
\setlist[enumerate]{noitemsep}

\algsetup{indent=2em}

\newcommand{\degsym}{\mathrm{deg}}
\newcommand{\idegsym}{\degsym_{\mathrm{i}}}
\newcommand{\odegsym}{\degsym_{\mathrm{o}}}
\newcommand{\udegsym}{\degsym_{\mathrm{u}}}
\newcommand{\indegree}{\idegsym}
\newcommand{\outdegree}{\odegsym}
\newcommand{\undegree}{\udegsym}
\newcommand{\degree}{\degsym}
\newcommand{\contraction}[1]{\mathrm{Cont}(#1)}
\newcommand{\treetag}{\mathbf{:{\mkern-5mu}t}}
\newcommand{\hybtag}{\mathbf{:{\mkern-5mu}h}}
\newcommand{\roottag}{\mathbf{:{\mkern-5mu}r}}

\newcommand{\muv}{\mu_{V}} 
\newcommand{\mue}{\mu_{E}} 
\newcommand{\dmue}{d_{\mue}} 
\newcommand{\dmuv}{d_{\muv}} 

\usepackage{tikz}

\newcommand\submittedtext{%
\footnotesize This work has been submitted to the IEEE for possible publication.
Copyright may be transferred without notice,
after which this version may no longer be accessible.}

\newcommand\submittednotice{%
\begin{tikzpicture}[remember picture,overlay]
\node[anchor=north,yshift=40pt] at (current page.south) {\fbox{\parbox{\dimexpr0.65\textwidth-\fboxsep-\fboxrule\relax}{\submittedtext}}};
\end{tikzpicture}%
}

\begin{document}

\title{A dissimilarity measure for semidirected networks}

\author{
  Michael~Maxfield${}^1$\orcidlink{0000-0003-3087-0085},
  \and
  Jingcheng~Xu${}^1$\orcidlink{0000-0003-2387-9188},
  \and
  C\'ecile An\'e${}^{1,2}$\orcidlink{0000-0002-4702-8217}
 \thanks{${}^1$ Department of Statistics, University of Wisconsin - Madison, USA.}%
 \thanks{${}^2$ Department of Botany, University of Wisconsin - Madison, USA.}
}

\markboth{}%
{}

\IEEEpubid{}

\maketitle
\submittednotice

\begin{abstract}
  Semidirected networks have received interest in evolutionary biology
  as the appropriate generalization of unrooted trees to networks,
  in which some but not all edges are directed.
  Yet these networks lack proper theoretical study.
  We define here a general class of semidirected phylogenetic networks,
  with a stable set of leaves, tree nodes and hybrid nodes.
  We prove that for these networks, if we locally choose the direction of one edge,
  then globally the set of directed paths starting by this edge
  is stable
  across all choices to root the network.
  We define an edge-based representation of semidirected phylogenetic networks
  and use it to define a dissimilarity between networks, which can be efficiently
  computed in near-quadratic time. Our dissimilarity extends the widely-used
  Robinson-Foulds distance on both rooted trees and unrooted trees.
  After generalizing the notion of tree-child networks to semidirected networks,
  we prove that our edge-based dissimilarity is in fact a distance on the space of
  tree-child semidirected phylogenetic networks.
\end{abstract}



\section{Introduction}

\IEEEPARstart{H}{istorical} relationships between species, virus strains or languages
are represented by phylogenies, which are rooted graphs, in which
the edge direction indicates the flow of time going forward.
Semidirected phylogenetic networks are to rooted networks what undirected trees
are to rooted trees.
They appeared recently, following studies showing that the root and the
direction of some edges in the network may not be identifiable,
from various data types
\citep{2016SolisLemusAne,2019Banos}.
Consequently, several methods to infer phylogenies from data aim to estimate
semidirected networks, rather than fully directed networks, such as
\texttt{SNaQ} \citep{2016SolisLemusAne},
\texttt{NANUQ} \citep{2019Allman_NANUQ},
\texttt{admixtools2} \citep{2023Maier},
\texttt{poolfstat} \citep{2022Gautier-poolfstat},
\texttt{NetRAX} \citep{2022netrax}, and
\texttt{PhyNEST} \citep{2023Kong-phynest}.

The theoretical identifiability of semidirected networks is receiving
increased attention
\citep{2021Gross_level1,2023MartinMoultonLeggett,2023XuAne,2024Ane-anomalies}
but graph theory is still at an early stage for this type of network
\citep[but see][]{2023LinzWicke}.
In particular, adapted distance and dissimilarity measures are lacking, as
are tools to test whether two phylogenetic semidirected networks are isomorphic.
These tools are urgently needed for applications.
For example, when an inference method is evaluated using simulations,
its performance is quantified by how often the inferred
network matches the true network used to simulate data, or how similar
the inferred network is to the true network.
Tools for semidirected networks would also help summarize a posterior
sample of networks output by Bayesian inference methods.
Even for a basic summary such as the posterior probability of a given
topology, we need to decide which semidirected networks are isomorphic,
in a potentially very large posterior sample of networks.

Unless additional structure is assumed, current methods
appeal to a naive strategy that considers all possible ways to root semidirected
networks, and then use methods designed for directed networks,
e.g. to check if the candidate rooted networks are isomorphic
or to minimize a dissimilarity between the two rooted networks across
all possible root positions.

In this work, we first generalize the notion of semidirected phylogenetic networks
in which edges are either of tree type or hybrid type,
such that any edge can be directed or undirected.
We relax the constraint of a single root (of unknown position).
Multi-rooted phylogenetic networks were recently introduced,
although for directed networks, to represent the history of
closely related and admixed populations
\citep{2019SoraggiWiuf}
or distant groups of species that exchanged genes nonetheless
\citep{2022Huber-forestbased,2025Huber-properforestbased}.
Our general definition requires care to define a set of leaves consistent
across all compatible directed phylogenetic networks.

For these semidirected networks,
we define an edge-based ``$\mu$-representation'' $\mue$,
extending the node-based $\mu$-representation, denoted here as $\muv$, by
\citet{Cardona09}.
The ``ancestral profile'' of a rooted network contain the same
information as the node-based representation $\muv$,
that is: the number of directed paths from each node to each leaf
\citep{2021Bai-ancestralprofile}.
So our edge-based representation $\mue$ also extends the notion
of ancestral profile to semidirected networks, in which ancestral
relationships are unknown between some nodes because the root is unknown.
Briefly, $\mue$'s information for an edge depends on whether
the edge direction is known or implicitly constrained by the direction of other
edges in the network. For example, if $N$ is a semidirected tree and an edge
is explicitly or implicitly directed, then $\mue$ associates the edge
to the cluster of taxa below it.
If instead the edge direction depends on the unknown placement of the root(s),
then $\mue$ associates the edge to the bipartition of the taxa
obtained by deleting the edge from the semidirected tree.
If $N$ has reticulations, $\muv$ uses $\mu$-vectors to generalize the notion of
clusters, storing the number of directed paths from a given node to each taxon.
Our extension $\mue$ associates a directed edge to the $\mu$-vector of its
child node, and an undirected edge to two $\mu$-vectors: one for each
direction that the edge can take.
To handle and distinguish semidirected networks with multiple roots,
$\mue$ also associates each root to a $\mu$-vector,
well-defined even when the exact root position(s) are unknown.
We then define a dissimilarity measure $\dmue$
between semidirected networks.

Our network representation $\mue$ can be calculated in polynomial time,
namely $\mathcal{O}(n|E|)$ where $n$ is the number of leaves and
$|E|$ is the number of edges.
The associated network dissimilarity $\dmue$
can then be calculated in $\mathcal{O}\bigl(|E|(n+\log|E|)\bigr)$ time.
It provides the first dissimilarity measure adapted to semidirected networks
(without iteratively network re-rooting) that can be calculated in polynomial time.
\citet{2023LinzWicke} also recently considered semidirected networks (with a
single root of unknown position). They showed that
``cut edge transfer'' rearrangements, which transform one network into another,
define a finite distance on the space of level-1 networks with a
fixed number of hybrid nodes.
This distance is NP-hard to compute, however, because it extends the
subtree prune and regraft (SPR) distance
on unrooted trees \citep{2008Hickey-SPRdist}.

Finally, we generalize the notion of tree-child networks to semidirected networks,
and prove that $\dmue$ is a true distance on the subspace of tree-child
semidirected networks, extending the result of \crv\
to semidirected networks using an edge-based representation.
On trees, this dissimilarity equals the widely used Robinson-Foulds distance
between \emph{unrooted} trees \citep{1981RobinsonFoulds}.
We provide concrete algorithms to construct $\mue$ from a given semidirected network,
and to reconstruct the network from $\mue$.

As a proper distance, $\dmue$ can decide in polynomial time if two tree-child
semidirected networks are isomorphic.
For rooted phylogenetic networks, the isomorphism problem is polynomially
equivalent to the general graph isomorphism problem, even if restricted
to tree-sibling time-consistent rooted networks \citep{2014Cardona}.
As semidirected networks include single-rooted networks,
the graph isomorphism problem for semidirected networks is necessarily more
complex. The general graph isomorphism problem was shown to be of
subexponential complexity \citep{2016Babai} but is not known to be
solvable in polynomial time.
Therefore, there is little hope of finding a dissimilarity that can be computed
in polynomial time, and that is a distance for
general semidirected phylogenetic networks.
Hence, dissimilarities like $\dmue$ offer a balance between computation time
and the extent of network space in which it can discriminate between distinct
networks.

\section{Basic definitions for semidirected graphs}

For a graph $G$ we denote its vertex set as $V(G)$ and its edge set as $E(G)$.
The subgraph induced by a subset of vertices $V' \subseteq V(G)$
is denoted as $G[V']$, and the edge-induced subgraph is denoted as $G[E']$
for a subset $E' \subseteq E(G)$.

We use the following terminology for a directed graph $G=(V,E)$.
For a node $v$, its \emph{in-degree} is denoted as $\indegree(v, G)$
and \emph{out-degree} as $\outdegree(v, G)$.
Its \emph{total degree} $\degree(v, G)$ is $\indegree(v, G) + \outdegree(v, G)$.
We may omit $G$ when no confusion is likely.
For $u,v \in V$, we write $u > v$ if there exists a directed path $u \leadsto v$.
A node $v$ is a
\emph{leaf} if $\outdegree(v) = 0$; $v$ is an \emph{internal} node
otherwise. We denote the set of leaves as $V_L(G)$.
A node $v$ is a \emph{root} if $\indegree(v) = 0$;
a \emph{tree node} if $\indegree(v) \leq 1$;
and a \emph{hybrid node} otherwise.
We denote the set of tree nodes as $V_T$ and the set of hybrid nodes as $V_H$.
An edge $(u,v) \in E$ is a \emph{tree edge} if $v$ is a tree node;
a \emph{hybrid edge} otherwise.
We denote the set of tree edges as $E_T$ and the set of hybrid edges as $E_H$.
A \emph{descendant} of $v$ is any node
$u$ such that $v > u$. 
A \emph{tree path} is a directed path consisting only of tree edges.
A \emph{tree descendant leaf} of $v$ is any leaf $u$
such that there exists a tree path $v \leadsto u$.
An \emph{elementary path} in $G$ is a directed path such
that the first node has out-degree~$1$ in $G$
and all intermediate nodes have in-degree and
out-degree~$1$ in $G$.
In the remainder, we use ``path'' to mean
``directed path'' for brevity, unless otherwise specified.

We now extend these notions to semidirected graphs.

\begin{defn}[\textbf{semidirected graph}]\label{semi directed graph}
  A \emph{semidirected graph} $N$ is a tuple $N = (V,E)$, where $V$ is
  the set of vertices, and $E = E_U \sqcup E_D$ is the set of edges, 
  $E_U$ being the set of undirected edges and $E_D$ the set
  of directed edges.
\end{defn}

Undirected edges in $E_U$ are denoted as $uv$ for some $u,v \in V$,
instead of the standard notation of $\{u, v\}$ for brevity.
Directed edges in $E_D$ are denoted as $(u,v)$ for some $u,v \in V$,
implying the direction from $u$ to $v$, with $u$ referred to
as the \emph{parent} of the edge, and $v$ as its
\emph{child}.
A directed graph is a semidirected graph with no undirected edges:
$E_U=\emptyset$.

For $v\in V(N)$, $\indegree(v, N)$ denotes
the number of directed edges with $v$ as their child,
$\outdegree(v, N)$ the number of directed edges with $v$ as their parent,
$\undegree(v, N)$ the number of undirected edges in $N$ incident to $v$, and
$\degsym = \idegsym + \odegsym + \udegsym$.
We may omit $N$ when no confusion is likely.
Furthermore,
$\mathrm{child}(v,N) = \{w \in V : (v,w) \in E_D\}$.
$N$ is \emph{binary} if $\degree(v)=1$ or $3$ for all nodes.
$N$ is \emph{bicombining} if $\indegree{v} = 2$ for all hybrid nodes.

A semidirected graph $N' = (V,E')$ is \emph{compatible} with another
semidirected graph $N = (V,E)$ if $N'$ can be obtained from $N$
by directing some undirected edges in $N$.

The \emph{contraction} of
$N$, denoted as $\contraction{N}$,
is the directed graph obtained by contracting every
undirected edge in $N$.
It is well defined, as it can be viewed as the quotient graph of $N$
under the partition that groups nodes connected by a series of undirected edges.
For $v\in V(N)$, $\contraction{v,N}$ is defined as the node in $\contraction{N}$
which $v$ gets contracted into.

In a semidirected graph $N$, a node $v \in V(N)$ is a \emph{root}
if it only has outgoing edges.
It is a \emph{tree node} if $\indegree(v, N) \leq 1$.
Otherwise, $v$ is called a \emph{hybrid node}.
The set of tree nodes is denoted as $V_T$ or $V_T(N)$ and the set of hybrid nodes
as $V_H$ or $V_H(N)$.
A \emph{tree edge} is an undirected edge,
or a directed edge whose child is a tree node.
A \emph{hybrid edge} is a directed edge whose child is a hybrid node.
We denote the set of tree edges as $E_T$ or $E_T(N)$
and the set of hybrid edges as $E_H$ or $E_H(N)$.

A semidirected \emph{cycle} is a semidirected graph if its undirected
edges can be directed so that it becomes a directed cycle. A semidirected
graph is \emph{acyclic} if it does not contain a semidirected cycle.  We refer
to directed acyclic graphs as DAGs and to semidirected acyclic graphs as SDAGs.

Next, we define a more stringent notion of compatibility
to maintain the classification of nodes and edges
as being of tree versus hybrid type,
illustrated in Figs.~\ref{fig:rootedpartners} and~\ref{fig:phylocompatible}.

\begin{defn}[\textbf{phylogenetically compatible, rooted partner, network}]
  \label{phylogenetically compatible}
  An SDAG $N'$ is \emph{phylogenetically compatible} with another SDAG
  $N$ if $N'$ is compatible with $N$ and $E_H(N') = E_H(N)$.
  A \emph{rooted partner} of $N$ is
  a DAG that is phylogenetically compatible with $N$.
  A \emph{multi-root semidirected network}, or \emph{network} for short, is
  an SDAG that admits a rooted partner.
\end{defn}

Note that if SDAG $N'$ is phylogenetically compatible with SDAG $N$,
then $V_T(N) = V_T(N')$ and $V_H(N) = V_H(N')$.
Note also the rooted partner of a binary network may not be binary by
the typical definition for rooted networks, since the root can have degree $3$
instead of $2$.

Not all SDAGs are networks
(see Fig.~\ref{fig:rootedpartners} for an example).
We are interested in networks rather than general SDAGs,
because a network can represent evolutionary history up to ``rerooting'',
as captured by its rooted partners.
Fig.~\ref{fig:rootedpartners} gives an example network
($N$, top) and 2 of its 4 rooted partners (bottom).

\begin{figure}
\centerline{\includegraphics{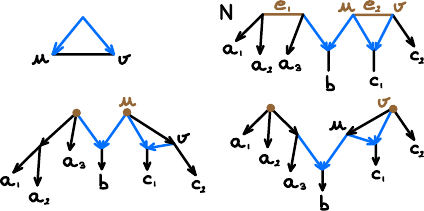}}
\caption{Examples of SDAGs and rooted partners.
  Top left: SDAG that is not a network, because it has no rooted partner:
  directing $uv$ would cause $u$ or $v$ to become a hybrid node, and result in
  a non-phylogenetically compatible DAG.
  Top right: $N$ is a network.
  Bottom: 2 of $N$'s 4 rooted partners.
  Each rooted partner has 2 roots (brown dots),
  one from each brown edge in $N$.
  Of the nodes incident to $e_2$ for example, either $u$ (bottom left) or
  $v$ (bottom right) can serve as root.
}\label{fig:rootedpartners}
\end{figure}

Traditionally, phylogenetic trees and networks are connected and with a single
root \citep{steel16_phylogeny}. We consider here
a broader class of graphs, allowing for multiple connected components and
multiple roots per connected component.
We also allow for non-simple graphs, that is, for multiple parallel
edges between the same two nodes $u$ and $v$,
directed and in the same direction for the graph to be acyclic.
With a slight abuse of notation, we keep referring to each parallel edge
as $(u,v)$.
Self-loops are not allowed as they would cause the graph to be cyclic.
The term ``rooted partner'' was introduced by \citet{2023LinzWicke}
in the context of traditional semidirected phylogenetic networks in which
the set of directed edges is precisely the set of hybrid edges,
and for which a rooted partner has a single root.

We will use the following results frequently. The first one is trivial.
\begin{prop}\label{prop:phylo-compat-only-directs-tree-edges}
  Let $N$ be an SDAG phylogenetically compatible with SDAG $N'$.
  Then $E_D(N) \setminus E_D(N') \subseteq E_T(N)$.
\end{prop}

\begin{prop}\label{prop:network-tree-edge}
  Let $N$ be a network, and $N'$ the semidirected graph obtained from $N$ by
  undirecting some of the tree edges of $N$.  Then $N'$ is a network, and $N$ is
  phylogenetically compatible with $N'$.
\end{prop}

\begin{proof}
  Let $A$ be the set of tree edges in $N$ that are undirected to obtain $N'$.
  We first consider the case when $A$ consists of a
  single edge $(u,v)$. We
  shall establish the following claims:
  \begin{enumerate}
  \item $N'$ is a network;
  \item the edges of $N'$ in $E(N) \setminus A$ retain their type
  (undirected, tree, or hybrid);
  \item $N$ is phylogenetically compatible with $N'$.
  \end{enumerate}

  For claim 1, suppose for contradiction that $N'$ is not an SDAG. Then there
  exists in $N'$ a semidirected cycle $C$ which contains $uv$. Let $G$
  be a rooted partner of $N$, and $C^+$ the subgraph made of the
  corresponding edges in $C$.  Since $C^+$ cannot be a directed cycle, there
  exist two hybrid edges $(a,h)$ and $(b,h)$ in $C^+$.
  Both are also directed in $N$ by phylogenetic compatibility.
  Because they are hybrid edges, they are distinct from
  $(u, v)$, so they are directed in $N'$ and $C$. This contradicts $C$
  being a semidirected cycle, showing that $N'$ is an SDAG.  Since it also
  admits $G$ as a rooted partner, $N'$ is a network.

  Claim 2 follows from the observation that the types of nodes (tree vs hybrid)
  in $N'$ stays the same.  Claim 3 follows from claim 2.

If $A$ contains multiple edges, then we iteratively undirect one edge at a time.
By the previous argument, the resulting graph at each step is a network
with which $N$ is phylogenetically compatible,
because phylogenetic compatibility is transitive.
Hence $N'$ is a network and $N$ is phylogenetically compatible with it.
\end{proof}

The next definition is motivated by the fact that phylogenies
are inferred from data collected at leaves, which are known
entities with labels (individuals, populations, or species),
whereas non-leaf nodes are inferred and unlabeled.

\begin{defn}[\textbf{unrooted, rooted and ambiguous leaves}]
A node $v$
in a network $N$
is a \emph{rooted leaf} in $N$ if
$v$ is a leaf in every rooted partner of $N$; and an \emph{unrooted leaf}
if it is a leaf in some rooted partner $N$.
We denote the set of rooted leaves and unrooted leaves as $V_{RL}(N)$
and $V_{UL}(N)$ respectively.
An \emph{ambiguous leaf} is a node in $V_{UL}(N) \setminus V_{RL}(N)$.
\end{defn}

\begin{figure}
\centering
\includegraphics{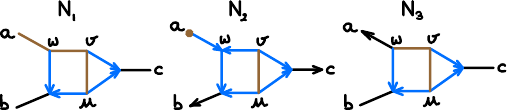}
\caption{Examples to illustrate definitions and notations.
  Directed edges ($E_D$) are shown with arrows; hybrid edges ($E_H$) in blue.
  Root components, whose nodes ($V_R$) can serve as roots, are shown in brown.
  $N_2$ is compatible but not phylogenetically compatible with $N_1$
  (e.g. $w$ is hybrid in $N_2$ and not in $N_1$).
  $N_3$ is phylogenetically compatible with $N_1$.
  Rooted leaves are $b$ and $c$ in $N_1$ and $N_2$;
  $a$ is an ambiguous leaf in $N_1$.
  $N_3$ has no ambiguous leaves so it is an $\mathcal{L}$-network
  on $\mathcal{L} = V_{UL}(N_3) = V_{RL}(N_3) = \{a,b,c\}$.
  $N_1$ has 1 root component.
  Its rooted partner rooted at $u$ is tree-child,
  but none of the others are (rooted at $a$, $w$ or $v$),
  so it is weakly tree-child.
  $N_2$ has 2 root components $R_1=\{a\}$ and $R_2=\{u,v\}$,
  and 2 rooted partners: one from the root choice
  $\rho(R_1)=a$, $\rho(R_2)=u$;
  and the other from the root choice
  $\rho(R_1)=a$, $\rho(R_2)=v$.
  $N_2$ is not tree-child in either sense, as none of its rooted partners are tree-child.
  $N_2 = \mathcal{C}(N_2)$ is complete.
  $N_3$ is not complete: in $\mathcal{C}(N_3)$, edges incident to $b$ and $c$
  are directed.
  }
\label{fig:phylocompatible}
\end{figure}

Clearly, $V_{RL}(N) \subseteq V_{UL}(N)$.
If $N$ is directed, then $V_{RL}(N) = V_{UL}(N) = V_L(N)$.
As we will see later in \Cref{lem:ULminusRL}, 
an ambiguous leaf
is of degree 1 in the undirected graph
obtained by undirecting all edges in $N$,
hence a leaf in the traditional sense.
In Fig.~\ref{fig:phylocompatible} for example,
$a$ is an ambiguous leaf in $N_1$ but a rooted leaf in $N_3$.

In \Cref{sec:SDAGproperties} we make no assumption regarding $V_{UL}(N)$
and $V_{RL}(N)$. But to extend $\mu$-vectors to semidirected graphs
in \Cref{sec:mu-rep}, we will need a stable set of leaves
across rooted partners.
For a vector of distinct labels $\mathcal{L}$, an $\mathcal{L}$-DAG is a DAG
whose leaves are bijectively labeled by $\mathcal{L}$.
To extend this definition we impose $V_{UL}(N) = V_{RL}(N)$ below,
but we will see (after \Cref{lem:ULminusRL})
that this requirement is less stringent than it appears.
\begin{defn}[\textbf{leaf-labeled network}]\label{def:labeled-network}
  A network $N$ is \emph{labeled in $\mathcal{L}$}
  and called an \emph{$\mathcal{L}$-network}, if
  $V_{UL}(N) = V_{RL}(N)$ and $V_{UL}(N)$ is bijectively labeled by elements
  of $\mathcal{L}$.
  The \emph{leaf set} of an $\mathcal{L}$-network $N$ is defined as
  $V_L(N) = V_{UL}(N)$.
\end{defn}

\noindent
For example, in Fig.~\ref{fig:phylocompatible} $N_3$ is an
$\mathcal{L}$-network with $\mathcal{L} = \{a,b,c\}$, while $N_1$ is not because
$V_{UL}(N_1) \neq V_{RL}(N_1)$.
$N_2$ is a $\{b,c\}$-network.

A rooted partner $G$ of an $\mathcal{L}$-network $N$ must be an
$\mathcal{L}$-DAG: since $V_{RL}(N) \subseteq V_L(G) \subseteq V_{UL}(N)$
generally, we have that $V_L(G) = V_L(N)$ and
$V_L(G)$ can inherit the labels in $N$.

Our main result concerns the class of tree-child graphs.
If $G$ is a DAG, it is \emph{tree-child} if every internal node $v$ of $G$
has at least one child that is a tree node \citep{2010HusonRuppScornavacca}.
Equivalently, $G$ is tree-child if every non-leaf node in $G$
has a tree descendant leaf \crv.
We extend this notion to semidirected networks,
illustrated in Fig.~\ref{fig:phylocompatible}.

\begin{defn}[\textbf{semidirected tree-child}]\label{semi-DAG tree-child}
  A network is \emph{weakly tree-child} if
  one of its rooted partners is tree-child.
  It is \emph{strongly tree-child}, or simply \emph{tree-child}, if
  all its rooted partners are tree-child.
\end{defn}
Since a DAG $G$ is a network with a single rooted partner,
$G$ is strongly and weakly tree-child
if and only if it is tree-child as a DAG.
In Fig.~\ref{fig:phylocompatible},
$N_3$ is weakly but not strongly tree-child:
of its 3 rooted partners, only one is tree-child.
Later in \Cref{prop:tree-child-iff}, we provide a characterization
easy to check without enumerating rooted partners.

\section{Properties of semidirected acyclic graphs}
\label{sec:SDAGproperties}

\begin{prop}\label{prop:SDG}
  A semidirected graph $N = (V,E)$ is acyclic if and only if the
  undirected graph induced by its undirected edges $E_U$ consists of trees only
  (i.e.\ is a forest),
  and $\contraction{N}$ is acyclic.
\end{prop}
\begin{proof}
  Let $N$ be a semidirected graph.
  Suppose that $N[E_U]$ is not a forest. Then there exists a compatible
  directed graph of $N[E_U]$ which contains a cycle that exists in a
  compatible directed graph of $N$, therefore $N$ is not acyclic.
  Next suppose
  that $\contraction{N}$ contains a
  cycle $C'$. Then there exists a
  compatible directed graph $G$ of $N$ which contains a cycle $C$ such
  that $C'$ is obtained from $C$ after contracting edges in $G$ that
  correspond to undirected edges in $N$, and so $N$ is not acyclic.

  Now suppose the graph induced by $E_U$ is a forest and that
  $\contraction{N}$ is acyclic. If $N$ is not acyclic, then by
  definition there exists a compatible directed graph
  $G$ that contains a cycle $C$. Since $\contraction{N}$ is acyclic, $C$ must
  contract into a single node in $\contraction{N}$. This implies that $C$ contains
  undirected edges only, hence is contained in
  $N[E_U]$, a contradiction. Therefore $N$ is acyclic.
\end{proof}

In a network $N$, a \emph{semidirected path} from $u_0$ to $u_n$ is a
sequence of
nodes $u_0u_1 \ldots u_n$, such that for
$i = 1, \ldots n$, either $u_{i-1}u_i$ or
$(u_{i-1},u_i)$ is an edge in $N$.
On $V(N)$ we define $v \lesssim u$ if there
is a semidirected path from $u$ to $v$, and
the associated equivalence relation: $u \sim v$ if $u \lesssim v$ and
$v \lesssim u$.
In Fig.~\ref{fig:phylocompatible}, $a \lesssim v$ in $N_1$ and $N_3$.
Also $v \lesssim a$ so $a \sim v$ in $N_1$, but not in $N_3$.
On equivalence classes,
$\lesssim$ becomes a partial order, with which we can define the following.
\begin{defn}[\textbf{undirected components, root components, directed part}]
\label{def:components}
  In a network $N$,
  an \emph{undirected component} is the subgraph
  induced by
  an equivalence class under $\sim$.
  A \emph{root component} of $N$ is
  an undirected component that is maximal under
  $\lesssim$.
  A root component is \emph{trivial} if it consists of a single node.
  The set of edges and nodes that are not in a root component is
  called the \emph{directed part} of $N$.  We denote the set of nodes
  in the directed part as $V_{DP}(N)$ and the set of edges $E_{DP}(N)$.
  We also denote the set of nodes in root components $V_R(N)$ and the set
  of edges $E_R(N)$.
\end{defn}

\medskip
The directed part of $N$ is generally not a subgraph:
it may contain an edge but not one of its
incident nodes.
In Fig.~\ref{fig:phylocompatible},
edges in the directed part $E_{DP}$ are those in black (tree edges)
and blue (hybrid edges). Edges in root components $E_R$ are in brown.
In $N_3$, $v$ is adjacent to the directed part,
but is
not in $V_{DP}(N_3)$.
$N_2$ has a trivial root component: $\{a\}$.
The following result
justifies the name given to the equivalence classes.

\begin{prop}
  \label{prop:same-component}
  For $u,v$ nodes in a network $N$, $u \sim v$ if and only if there
  is an undirected path between $u$ and $v$.
\end{prop}
\begin{proof}
  Consider $u$ and $v$ connected by an undirected path. 
  Since this path does not contain any directed edge, clearly, $u \lesssim v$ and 
  $v \lesssim u$, hence $u \sim v$.

  Now suppose $u \sim v$. By definition, there exists
  a semidirected path $p_{uv}= u_0u_1 \ldots u_n$ from $u=u_0$ to $v=u_n$ and
  a semidirected path $p_{vu}$ from $v$ to $u$.
  If $p_{vu} = u_n u_{n-1}\ldots u_0$ then all edges in these paths must be
  shared and undirected.
  This is because $N$ is acyclic, in case $p_{uv}$ and $p_{vu}$
  have distinct edges incident to $u_i$ and $u_{i+1}$.
  Then, there is an undirected path between $u$ and $v$ as claimed.
  Otherwise, there exists $i_1\geq 0$ and $i_2 > i_1+1$ such that
  $p_{vu}$ is the concatenation of semidirected paths
  $u_n u_{n-1}\ldots u_{i_2}$;
  $p_{u_{i_2} u_{i_1}}$ from $u_{i_2}$ to $u_{i_1}$
  not containing any $u_j$ for $i_1<j<i_2$; and
  $u_{i_1} u_{i_1-1}\ldots u_0$.
  Then, the concatenation of $p_{u_{i_2} u_{i_1}}$ with
  $u_{i_1} u_{i_1+1} \ldots u_{i_2}$
  forms a
  semidirected cycle.
  Since $N$ is acyclic, this case cannot occur.
\end{proof}

We now characterize undirected components as
the undirected trees in the forest induced by $N$'s undirected edges.

\begin{prop}\label{prop:M-forest}
  Let $N$ be a network and $F$ the graph induced by the
  undirected edges of $N$.  Then $F$ is a forest where:
  \begin{enumerate}
  \item each tree corresponds to an undirected component of $N$, and
  has at most one node $v$ with $\indegree(v, N) \geq 1$;
  \item the root components of $N$ are exactly the trees without such nodes,
  and they contain tree nodes only.
  \end{enumerate}
\end{prop}

\begin{proof}
  By \Cref{prop:SDG}, $F$ is a forest.
  By \Cref{prop:same-component}, each tree in $F$ is an undirected component.
  If a tree $T$ in $F$ had
  more than one node $v$ with $\indegree(v, N) \geq 1$, it would be
  impossible to direct the edges in $T$ without making one of them hybrid,
  contradicting the existence of a rooted partner of $N$.

  For the second claim, let $T$ be a tree in $F$.
  Note that $(u,v)\in E_D$ implies that
  $u \not\sim v$
  and $v$'s equivalence class is not maximal under $\lesssim$.
  So if $T$ is maximal under $\lesssim$, then all nodes $v$ in $T$ have $\indegree(v, N)=0$,
  which implies that $v$ is a tree node.
  Conversely, if $T$ is not maximal, then there exist nodes $v$ in $T$,
  $u\gtrsim v$ in a different tree of $F$, and a semidirected path
  from $u$ to $v$ containing a directed edge. Taking the last directed edge
  on this path gives a node $v'$ in $T$ with $\indegree(v',N)\geq 1$.
\end{proof}

\begin{lem}\label{undirected before directed}
  Let $N$ be a network and $G$ a rooted partner of $N$.
  Let $v \in V_R(N)$. If $v < u$ in $G$, then $u \in V_R(N)$.
\end{lem}
\begin{proof}
  Let $v < u$ in $G$. Then $v \lesssim u$ in $N$.
  By Definition~\ref{def:components}, $u \lesssim v$ in $N$ as well.
  Therefore, $u$ belongs to the same root component as $v$, hence $u \in V_R(N)$.
\end{proof}

\begin{prop}\label{prop:fixeddirection-in-DP}
  Let $N$ be a network. Let $G_1$ and $G_2$ be rooted partners of $N$.
  Then
  edges in $E_{DP}(N)$ have the same direction in $G_1$ and $G_2$.
\end{prop}
\begin{proof}
  Let $T$ be an undirected component of $N$ that is not a root component.
  We need to show that $T$'s edges have the same direction in $G_1$ and $G_2$.
  By \Cref{prop:M-forest}, $T$ is an undirected tree in $N$, and has exactly one
  node $v_0$ with an incoming edge in $N$.
  Then $v_0$ must be the root of $T$ in both $G_1$ and $G_2$,
  for $G_1$ and $G_2$ to be phylogenetically compatible with $N$,
  which completes the proof.
\end{proof}

\begin{defn}[\textbf{root choice function}]
  Let $N$ be a network, and $\mathcal{R}$ be the set of root components
  of $N$.  A \emph{root choice function} of $N$ is a function
  $\rho: \mathcal{R} \to V(N)$ such that for a root component
  $T \in \mathcal{R}$, $\rho(T) \in V(T)$.
  In other words, $\rho$
  chooses a node for each root component.
\end{defn}

Conceptually, a semidirected network represents uncertainty
about the root(s) position.
Next, we show that to resolve uncertainty,
exactly 1 node from each root component must be chosen as root,
and this choice can be made independently across root components.
In other words, root choice
functions are in one-to-one correspondence with rooted partners.
As a result, all rooted partners have the same number of roots:
the number of root components.
In Fig.~\ref{fig:rootedpartners}, $N$ has
2 root components each with 2 nodes, hence $2\times 2 = 4$ rooted partners.
\begin{prop}
  \label{prop:network-rooting}
  Given a root choice function $\rho$ of a network $N$, there exists a
  unique rooted partner $N^+_\rho$ of $N$ such that the set of roots in $N^+_\rho$
  is the image of $\rho$.
  Conversely, given any rooted partner $G$ of $N$, there exists a unique root 
  choice function $\rho$ such that $G = N^+_\rho$.
\end{prop}
\begin{proof}
  Let $G_0$ be a rooted partner of $N$.
  For a root choice function $\rho$, let $N^+_\rho$ be the graph
  compatible with $N$ obtained by directing edges in $E_{DP}(N)$ as they are in $G_0$,
  and away from $\rho(T)$ in any root component $T$
  (which is possible by \Cref{prop:M-forest}).
  $N^+_\rho$ is a DAG because $N$ is acyclic.
  To prove that $N^+_\rho$ is phylogenetically compatible with $N$
  (and hence a rooted partner), we need to show that $E_H(N^+_\rho)=E_H(G_0)$.
  Since all edges in $N$ incident to both $V_R=V_R(N)$ and
  $V_{DP}=V_{DP}(N)$ are directed
  from $V_R$ to $V_{DP}$, they have the same direction in $N^+_\rho$ and $G_0$.
  Therefore nodes in $V_{DP}$ and edges in $E_{DP}(N)$ are of the same type in
  $N^+_\rho$ and $G_0$ (and $N$).
  Furthermore, by construction and~\Cref{prop:M-forest}, all edges in
  $E_R(N)$ remain of tree type in both $N^+_\rho$ and $G_0$.
  Hence $N^+_\rho$ is a rooted partner of $N$.
  Finally, the root set of $N^+_\rho$ is the image of $\rho$ because
  a root component's root is a root of the full network
  (from $\indegree(v, N) = 0$ for any $v\in V_R$)
  and because $V_{DP}$ cannot contain any root of $N^+_\rho$
  (by \Cref{prop:M-forest} again).

  To prove that $N^+_\rho$ is unique,
  let $G$ be a rooted partner of $N$ whose set of roots is the image of $\rho$.
  By \Cref{prop:fixeddirection-in-DP},
  edges in $E_{DP}(N)$ have the same direction in $G$ and $N^+$.
  For a root component $T$, $G[V(T)] = N^+_\rho[V(T)]$ because
  $T$ is an undirected tree in $N$, rooted by the same $\rho(T)$
  in both $G$ and $N^+_\rho$. Therefore $G=N^+_\rho$.

  Let $G$ be a rooted partner of $N$.
  By \Cref{prop:M-forest} and phylogenetic compatibility,
  $G$ must have exactly one root in each root component $T$.
  Define $\rho$ such that
  $\rho(T)$ is this root of $G$ in $T$.
  By the arguments above, $G$ cannot have any root in $V_{DP}$,
  and then $G=N^+_\rho$.
\end{proof}

We can define the following
thanks to \Cref{prop:fixeddirection-in-DP}:

\begin{defn}[\textbf{network completion}]\label{def:semiDAG-completion}
  The \emph{completion} $\mathcal{C}(N)$ of a network $N$ is the semidirected graph
  obtained from $N$ by directing every undirected edge in its directed part,
  as it is in any rooted partner of $N$.  More precisely, let $G$ be
  a rooted partner of $N$.
  We direct $uv \in E_{DP}(N)$ as $(u, v)$ in $\mathcal{C}(N)$ if
  $(u, v) \in E(G)$.
  A network $N$ is \emph{complete} if $\mathcal{C}(N) = N$.
\end{defn}

\begin{prop}\label{prop:completioniscompatible}
  For a network $N$, $\mathcal{C}(N)$ is a network and phylogenetically
  compatible with $N$.
\end{prop}

\begin{proof}
  From \Cref{prop:network-rooting}, any rooted partner of $N$
  is of the form $N^+_{\rho}$. As seen in the proof of \Cref{prop:network-rooting},
  $\mathcal{C}(N)$ and $N^+_{\rho}$ differ in root components only:
  if $e\in E_R(N)$ then $e$ is undirected in $N$ and in $\mathcal{C}(N)$,
  directed in $N^+_{\rho}$, and is a tree edge in all. Therefore
  $\mathcal{C}(N)$ is phylogenetically compatible with $N$ and is a network
  (admitting $N^+_{\rho}$ as rooted partner).
\end{proof}

\begin{remark}\label{rem:alg-completion}
Propositions~\ref{prop:M-forest} and~\ref{prop:network-rooting}
yield practical algorithms. Finding the root
components requires only traversing the network $N$, tracking the forest $F$
of undirected components, and which nodes have nonzero in-degree.
Computing $\mathcal{C}(N)$ then consists in
directing the edges away from such a node in each tree of $F$, if
one exists.
In particular, in a single traversal of $N$ we can construct a rooted partner
$G$, record the roots of $G$, record the edges of $G$ that were in
root components of $N$, and for each such edge record the corresponding root.
To do this, for each tree $T$ in $F$ that is a root component, we arbitrarily
choose and record a node $u$ as root, direct the edges in $T$ away from
$u$, record these edges as belonging to a root component of $N$, and for all
these edges record $u$ as the corresponding root.  We then direct the rest of
the undirected edges the same way as when computing the completion.
We shall use this in \Cref{alg:edge-mu-rep} later.
\end{remark}

With as many directed edges as can be possibly
implied by the directed edges in $N$, $\mathcal{C}(N)$ is the network that
we are generally interested in. We define
phylogenetic isomorphism between networks based on their completion.

\begin{defn}[\textbf{network isomorphism}]\label{def:isomorphism}
  $\mathcal{L}$-networks $N$ and $N'$ are
  \emph{phylogenetically isomorphic}, denoted by $N \cong N'$,
  if $\mathcal{C}(N)$ and
  $\mathcal{C}(N')$ are isomorphic as semidirected graphs,
  with an isomorphism that preserves the leaf labels.
\end{defn}

In Fig.~\ref{fig:phylocompatible} for example,
$N_2$ is complete, but $N_3$ is not.
$N_1$ and $N_3$ are not phylogenetically isomorphic, because the edge incident
to $a$ remains undirected in the completion $\mathcal{C}(N_1)$.
$N_2$ is not phylogenetically compatible with $N_1$ or $N_3$,
and not phylogenetically isomorphic to either.

\begin{lem}\label{lem:directed-part-edge-and-node}
  Let $N$ be a complete network.
  There exists a directed edge $(u,v)$ in $N$
  if and only if $v \in V_{DP}(N)$.
\end{lem}
\begin{proof}
  If $v \in V_{DP}(N)$, its undirected component $U$ has no
  undirected edges by \Cref{def:semiDAG-completion}.
  Then by \Cref{prop:M-forest}, $U=\{v\}$ and $\indegree(v) \geq 1$,
  so $v$ is the child of some directed edge.
  Conversely,
  if there exists $(u,v)\in E(N)$ then $v$'s equivalence class
  is not maximal
  and $v \in V_{DP}(N)$.
\end{proof}

Using 
$\mathcal{C}(N)$, we can now decide if a network is
weakly or strongly tree-child in a single
traversal, thanks to the following.

\begin{prop}\label{prop:tree-child-iff}
  Let $N$ be a complete network,
  $\mathcal{R}$ its set of root components, and
  $W_0$ the set of nodes that form trivial root components.
  For $T \in \mathcal{R}$, let $W_1(T)$ be the set of nodes $u$ in $T$
  adjacent to $V_{DP}(N)$
  with $\undegree(u) = 1$ and without a tree-child in $N$.
  $N$ is weakly (resp. strongly) tree-child if and only if
  every non-leaf node
  in $V_{DP}(N) \cup W_0$ has a tree child in $N$;
  and for every $T\in \mathcal{R}$, $|W_1(T)|\leq 1$ (resp. $W_1(T)=\emptyset$).
\end{prop}

\noindent
If $N$ is a DAG, then
$V_{DP}(N) \cup W_0$ is the full node set
and we simply recover the tree-child definition.
Recall that children are defined using directed edges only.
For example, take $N$ in Fig.~\ref{fig:rootedpartners}.
In $\mathcal{C}(N)$ the edges incident to $b$ and $c_1$ are directed;
$e_i$ remains undirected and forms a root component $T_i$ ($i=1,2$).
$W_1(T_1)=\emptyset$
but $W_1(T_2)=\{u\}$, because $u$ is incident to exactly 1 undirected edge
and 2 outgoing hybrid edges.
By \Cref{prop:tree-child-iff}, $N$ is weakly tree-child.
Indeed, its partners rooted at $u$ are tree-child
(e.g. Fig.~\ref{fig:rootedpartners} bottom left)
but its partners rooted at $v$ are not (e.g.\ bottom right).
The proof below shows that, more generally, a weakly tree-child
network has at most one `problematic' node in each root component, and this
node must serve as root for a rooted partner to be tree-child.

\begin{proof}[Proof of \Cref{prop:tree-child-iff}]
  We first characterize which rooted partners are tree-child.
  Suppose that $N$ is complete, and that
  every non-leaf node in $V_{DP}(N) \cup W_0$ has a tree child in $N$.
  For a root choice function $\rho$,
  we claim that the rooted partner $N^+_\rho$ is tree-child if and only if
  $W_1(T) \subseteq \{\rho(T)\}$ for every $T \in \mathcal{R}$.
  To prove this claim, consider a node $u$ in $N$.
  If $u$ is in $V_{DP}(N) \cup W_0$ then
  $u$ has the same children in $N$ as in any rooted partner, so
  its tree child in $N$ is also its tree child in $N^+_\rho$.
  Otherwise, $u$ is in some $T\in\mathcal{R}$,
  $T$ is non-trivial and $\undegree(u) \geq 1$.
  If $\undegree(u) \geq 2$, then at least one of its neighbor in $T$ is
  its tree child in $N^+_\rho$.
  If $\undegree(u) = 1$ and is not adjacent to $V_{DP}(N)$, then
  $\outdegree(u)=0$ and $u$ is a leaf in $N^+_\rho$ or its unique neighbor
  is its tree child in $N^+_\rho$.
  If $u$ has a tree child $w$ in $N$, then $w$ is also its tree child in $N^+_\rho$.
  Otherwise, $u\in W_1(T)$.
  Let $v$ be its unique neighbor in $T$.
  If $\rho(T)\neq u$ then $v$ is the parent of $u$ in $N^+_\rho$,
  $u$ has the same children in $N^+_\rho$ as in $N$,
  so $u$ has no tree child in $N^+_\rho$ (by definition of $W_1(T)$).
  If $\rho(T)=u$ then $v$ is a child of $u$ in $N^+_\rho$, and since $v$
  is a tree node (because $v\in T$) $u$ has a tree child in $N^+_\rho$.
  Overall, we get that $N^+_\rho$ is tree-child if and only if
  any node in any $W_1(T)$ is a root,
  which proves the claim.

  For the first direction of \Cref{prop:tree-child-iff},
  suppose that $N$ is complete and weakly tree-child.
  If $u$ is a non-leaf node in $V_{DP}(N) \cup W_0$,
  then $u$ has the same children in $N$ as in any rooted partner,
  so $u$ has a tree child in $N$. Also, we can apply our claim.
  Since $N^+_\rho$ is tree-child for some $\rho$, by our claim we get
  $W_1(T) \subseteq \{\rho(T)\}$ which implies that $|W_1(T)| \leq 1$ for
  each $T \in \mathcal{R}$.
  Suppose further that $N$ is strongly tree-child.
  If $W_1(T) \neq \emptyset$ for some $T \in \mathcal{R}$ then $T$ is nontrivial,
  we can choose $\rho(T)=v \notin W_1(T)$ for which $N^+_\rho$ is not tree-child,
  a contradiction. This proves the first direction.

  For the second direction, assume that every non-leaf node
  in $V_{DP}(N) \cup W_0$ has a tree child in $N$ and
  $|W_1(T)|\leq 1$ (resp. $W_1(T)=\emptyset$) for every $T\in \mathcal{R}$.
  Then $N$ admits at least one tree-child rooted partner 
  (setting $\rho$ such that $\{\rho(T)\} = W_1(T)$ for any $T$ with
  $W_1(T) \neq \emptyset$) and $N$ is weakly tree-child.
  Further, if $W_1(T)=\emptyset$ for all $T$,
  then $N^+_\rho$ is tree-child for any root choice function $\rho$,
  and $N$ is strongly tree-child.
\end{proof}

Finally, we characterize unambiguous leaves quickly.

\begin{lem}\label{lem:ULminusRL}
  In a network $N$, $v$ is an ambiguous leaf
  if and only
  if $v \in V_R(N)$, $\undegree(v) = 1$, and $\outdegree(v) = \indegree(v) = 0$.
\end{lem}

\begin{proof}
  Let $v$ be an ambiguous leaf. Then $\outdegree(v) = 0$.
  By~\Cref{prop:fixeddirection-in-DP}, $v$ is in $V_R(N)$
  so $\indegree(v) = 0$.
  Finally, if $\undegree(v) = 0$, then
  $v$ is isolated, and a rooted leaf.
  If $\undegree(v) \geq 2$, then in any rooted
  partner one of the incident edges is directed away from $v$, and $v$ is never
  a leaf.  Therefore $\undegree(v) = 1$.

  Conversely, let $v \in V_R(N)$
  be incident to exactly one edge, $uv$.
  By \Cref{prop:network-rooting}, we can find a rooted partner where $v$
  is a non-leaf root as well as a rooted partner where $u$ is a root and $v$ is
  a leaf.  Hence $v$ is an ambiguous leaf.
\end{proof}

\begin{remark}
We argue that requiring $V_{UL}(N) = V_{RL}(N)$
for $N$ to be an $\mathcal{N}$-network is reasonable in practice.
By \Cref{lem:ULminusRL}, an ambiguous leaf $x$ is
in a root component and incident to only one undirected edge.
The ambiguity is whether $x$ becomes a leaf or a root in
a rooted partner.
In practice, one knows which nodes are supposed to be
leaves, with a label and data
\citep{2004Felsenstein,2022Neureiter_contacTrees}.
Then one can, for each root component, either direct all
incident edges to ambiguous leaves towards them, making them rooted leaves
(as in Fig.~\ref{fig:phylocompatible}, compare $N_1$ and $N_3$), or pick one
ambiguous leaf to serve as root for that component
(as in $N_2$, Fig.~\ref{fig:phylocompatible})
and direct edges accordingly,
turning the remaining ambiguous leaves into rooted leaves.
This yields a network with $V_{UL} = V_{RL}$.
By \Cref{prop:M-forest}, this can be done in $\mathcal{O}(|E|)$ where
  $|E|$ is the number of edges.
\end{remark}

\section{Vectors and Representations}
\label{sec:mu-rep}

Formally a multiset is a tuple $(A, m)$, where $A$ is a set and
$m$: $A \to \mathbb{Z}^+$ gives the multiplicity.
We consider two multisets $(A, m_A)$ and $(B, m_B)$ to be the same if $m_A|_{A \setminus B} = 0$,
$m_B|_{B \setminus A} = 0$, and $m_A|_{A \cap B} = m_B|_{A \cap B}$.
To simplify notations, we
use $\lbag \rbag$ to denote a multiset by
enumerating each element as many times as its multiplicity.
For example, $A = \lbag a, a, b \rbag$ contains $a$ with multiplicity 2 and $b$
with multiplicity 1.
For brevity, we identify a multiset $(A, m)$ with the set $A$ if $m\equiv 1$,
e.g. $\lbag a, b, c \rbag = \{a, b, c\}$.
We adopt the standard notion of sum and difference for multisets.
The symmetric difference between multisets
is defined as $A \triangle B = (A - B) + (B - A)$.

In what follows, we consider a vector of distinct labels $\mathcal{L}$,
whose order is arbitrary but fixed, as it will determine the order of
coordinates in all $\mu$ vectors and representations.
For an $\mathcal{L}$-DAG $G$ and for node $v\in V(G)$,
the $\mu$-vector of $v$ is defined as the tuple
$\mu(v,G) = (\mu_1(v),\ldots,\mu_n(v))$
where $n$ is the number of labels in $\mathcal{L}$ and
$\mu_i(v)$ is the number of paths in $G$ from $v$ to
the leaf with the $i^\mathrm{th}$ label in $\mathcal{L}$.
As in \crv, the partial order $\geq$ between $\mu$-vectors is the coordinatewise
order. Namely, for $m=(m_1, \ldots, m_n)$ and $m'=(m_1', \ldots, m_n')$,
$m \geq m'$ if $m_i \geq m_i'$ for
all $i = 1, \ldots, n$.
If $m\not\leq m'$ and $m\not\geq m'$ then $m$ and $m'$ are
\emph{incomparable}.
The node-based $\mu$-representation of $G$ from \crv, denoted as $\muv(G)$, is
defined as the multiset $\lbag \mu(v, G): v\in V(G) \rbag$.
Algorithm~1 in \crv\ computes $\muv(G)$ recursively in post-order thanks to
the following property, which is a slight extension of Lemma~4 in
\crv\ allowing for parallel edges
by summing over child edges instead of child nodes.

\begin{lem}
  \label{lem:c9-4}
  Let $G$ be a DAG and $u$ a node in $G$. Then
  $$ \mu(u,G) = \sum_{v \in \mathrm{child}(u,G)}\;\;
     \sum_{(u,v) \in E(G)} \mu(v,G) \;.$$
\end{lem}

We will make frequent use of the following result, whose original proof easily
extends to DAGs with parallel edges thanks to \Cref{lem:c9-4}.
It is an extension of Lemma~5 of \crv\ stating the assumption
used in the proof by \crv, which is weaker than requiring a tree-child DAG.
\begin{lem}
  \label{lem:c9-5}
  Let $G$ be
  an $\mathcal{L}$-DAG and
  $u, v$ two nodes in $G$.
  \begin{enumerate}
  \item If there exists a path $u \leadsto v$, then $\mu(u, G) \geq \mu(v, G)$.
  \item If $\mu(u, G) > \mu(v, G)$ and
  if $v$ has a tree descendant leaf,
  then there exists a path $u \leadsto v$.
  \item If $\mu(u, G) = \mu(v, G)$ and
    if $v$ has a tree descendant leaf,
    then $u, v$ are connected by an elementary path.
  \end{enumerate}
\end{lem}
\noindent
Other results in \crv\ similarly hold when
parallel edges are allowed, such as their Theorem~1 on tree-child networks
(which must be non-binary if they have parallel edges).

The rest of the section is organized as follows. In part~\ref{sec:EBMR} we define
the edge-based $\mu$-representation for $\mathcal{L}$-networks, with
\Cref{alg:edge-mu-rep} to compute it.
Part~\ref{sec:EBMR-treechild} presents properties of this $\mu$-representation
for tree-child networks, and
part~\ref{sec:reconstruct} describes
how to reconstruct a complete tree-child $\mathcal{L}$-network from its
edge-based $\mu$-representation.
The networks in Fig.~\ref{fig:compare2nets} are used as
examples throughout.

\subsection{Edge-based $\mu$-representation}\label{sec:EBMR}

We first extend the notion of $\mu$-vectors to nodes in the directed part of an
$\mathcal{L}$-network.

\begin{prop}\label{prop:directed-part-mu-vec}
  Let $N$ be an $\mathcal{L}$-network, and $v \in V_{DP}(N)$.
  Then the set of directed paths starting at $v$,
  and consequently $\mu(v, G)$, are the same for any rooted partner $G$ of $N$.
\end{prop}

Using the proposition above, the following is well-defined.
\begin{defn}[\textbf{$\mu$-vector of a node in the directed part}]
  \label{def:directed-part-mu-vec}
  Let $N$ be an $\mathcal{L}$-network and $G$ any rooted partner of $N$.
  For $v \in V_{DP}(N)$, we define $\mu(v, N) = \mu(v, G)$.
\end{defn}

In Fig.~\ref{fig:compare2nets} for example, $e_5$ is in the
directed part of $N$.
Applying \Cref{lem:c9-4} recursively on $h_1$, $h_2$ and
their parent in ${\mathcal C}(N)$,
we get $\mu(e_5,N) = (0{,}0{,}0{,}0{,}0, 1{,}1{,}0)$
with leaf order given by $\mathcal{L}=(a_1,a_2,b,c,d,h_1,h_2,h_3)$.
All 3 hybrid edges have the
same $\mu$-vector in $N$: $(0{,}0{,}0{,}0{,}0, 1{,}1{,}1)$.

\begin{proof}[Proof of \Cref{prop:directed-part-mu-vec}]
  Let $G$ be a rooted partner of $N$.  Let $u_1\ldots u_n$ ($n \geq 1$), where
  $u_1 = v$, be a directed path starting at $v$ in $G$.  We claim
  $u_i, i = 1, \ldots, n$ are all in $V_{DP}(N)$: Otherwise, we can find $i$ such
  that $u_i \in V_{DP}(N)$ and $u_{i+1} \in V_R(N)$.  Since
  $(u_i, u_{i+1}) \in E(G)$, either $u_iu_{i+1}$ or $(u_i, u_{i+1})$ is in
  $E(N)$.  By \Cref{prop:same-component}, this implies either $u_i \in V_R(N)$
  or $u_{i+1} \in V_{DP}(N)$, a contradiction.
  Therefore any directed paths from $v$ in $G$ lies entirely in $G[V_{DP}(N)]$.
  The conclusion then follows from \Cref{prop:fixeddirection-in-DP}.
\end{proof}

Now we turn to root components.
Here the $\mu$-vector for a node
is not well-defined as it varies depending on the rooted partner.
It turns out that if locally we
  choose the direction of an edge $uv$, say $(u, v)$, then globally the set of
  directed paths from $v$ across all rooted partners are the same, and
  consequently the $\mu$-vector of $v$ is fixed.
We now formalize this.

\begin{prop}\label{prop:directional-mu-vec}
  Let $N$ be an $\mathcal{L}$-network with $uv \in E_R(N)$.  Then
  the set of directed paths starting at $v$, and consequently $\mu(v, G)$,
  are the same for any rooted partner $G$ of $N$ where $uv$
  is directed as $(u, v)$, and there always exists such a rooted partner. 
\end{prop}

Using this proposition, we can define the following.

\begin{defn}[\textbf{directional $\mu$-vector}]\label{def:directional-mu-vec}
  Let $N$ be an $\mathcal{L}$-network, $uv \in E_R(N)$, and $G$ any rooted
  partner of $N$ where $uv$ is directed as $(u, v)$.  We call $\mu(v, G)$ the
  \emph{directional $\mu$-vector of $(u, v)$}, and write it as $\mu_d(u, v, N)$,
  or $\mu_d(u, v)$ if $N$ is clear in the context.
\end{defn}

\begin{figure}
\centering \includegraphics{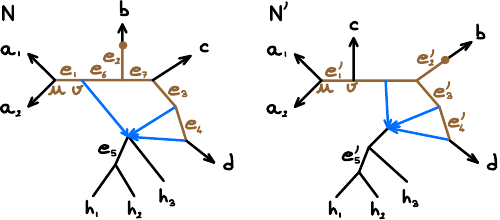}
\caption{Tree-child $\mathcal{L}$-networks on
  $\mathcal{L}=(a_1,a_2,b,c,d,h_1,h_2,h_3)$ with one root component.
  Directed edges are shown with arrows.
  Black: tree edges in the directed part, leading to $A_1$ in \Cref{alg:edge-mu-rep}.
  Blue: hybrid edges, leading to $A_2$ in \Cref{alg:edge-mu-rep}.
  Brown: edges in the root component ($E_R$),
  leading to $A_4$ in \Cref{alg:edge-mu-rep}.
  For $i\leq 5$, edge $e_i$ in $N$ (left) and $e'_i$ in $N'$ (right)
  share the same $\mu$-vector set.
  The multisets $\mue(N)$ and $\mue(N')$ have 17 elements in common:
  8 corresponding to edges incident to leaves,
  5 from edges $e_i$ and $e'_i$ for $i\leq 5$,
  3 from hybrid edges, and 1 root $\mu$-vector set.
  $\mue(N)$ has 2 unique elements
  (from $e_6$ and $e_7$)
  and $\mue(N')$ has 3, so $\dmue(N,N')=5$.
  See the Appendix
  for details.
}\label{fig:compare2nets}
\end{figure}

In Fig.~\ref{fig:compare2nets} for example, $e_1=uv$ is in the
root component of $N$, with directional $\mu$-vectors:
$\mu_d(v,u,N) = (1{,}1{,}0{,}0{,}0, 0{,}0{,}0)$ and
$\mu_d(u,v,N) = (0{,}0{,}1{,}1{,}1, 3{,}3{,}3)$.
In $N'$, $e'_1$ has these same directional $\mu$-vectors.

Before the proof we establish a lemma.

\begin{lem}\label{lem:direct-one-edge}
  Let $N$ be a network with $uv$ an edge in some root component.  Let $N'$ be the
  semidirected graph obtained from $N$ by directing $uv$ as $(u, v)$.  Then $N'$
  is a network phylogenetically compatible with $N$.
\end{lem}

\begin{proof}
  Let $G$ be a rooted partner of $N$ where $u$ is a root.  Let $A = E(G) \setminus
  E(N)$ be the set of edges that are directed in $G$ but not in $N$.  By
  phylogenetic compatibility, $A$ consists of tree edges only.  $A$ also
  contains $(u, v)$.

  Since from $G$ we get back $N'$ if we undirect all the edges in
  $A \setminus \{(u, v)\}$, by \Cref{prop:network-tree-edge}, $G$ is
  phylogenetically compatible with $N'$.  Then $E_H(N') = E_H(G) = E_H(N)$ and $N'$
  is phylogenetically compatible with $N$.
\end{proof}

\begin{proof}[Proof of \Cref{prop:directional-mu-vec}]
  The existence of a rooted partner where $uv$ is directed as $(u,v)$ follows
  from \Cref{prop:network-rooting}. 

  Let $G_1$ and $G_2$ be rooted partners of $N$ such that 
  $uv$ is directed as $(u,v)$ in both.
  Let $N'$ be the semidirected graph obtained from $N$ by directing $uv$ as $(u,v)$. 
  By Lemma~\ref{lem:direct-one-edge}, $N'$ is a network phylogenetically compatible 
  with $N$. 
  Thus, $G_1$ and $G_2$ are rooted partners of $N'$. 
  Since $\indegree(v, N') \geq 1$, $v \in V_{DP}(N')$ by \Cref{prop:M-forest}.
  The conclusion then follows from \Cref{prop:directed-part-mu-vec}.
\end{proof}

In the next lemma and definition, we associate each root component (rather than
a node or edge) to a $\mu$-vector.

\begin{lem}\label{lem:complementary-mu-vec}
  Let $N$ be an $\mathcal{L}$-network, and $T$ a root component of $N$.
  Then the $\mu$-vector $\mu(\rho(T), N^+_\rho)$ is independent of
  the root choice function $\rho$.  
  Furthermore, if $uv$ is an edge in $T$, this $\mu$-vector is equal to
  $\mu_d(u, v) + \mu_d(v, u)$.
\end{lem}

With this lemma we define the following:

\begin{defn}[\textbf{root $\mu$-vector}]\label{def:complementary-mu-vec}
  Let $N$ be an $\mathcal{L}$-network and $T$ a root component of $N$.
  The \emph{root $\mu$-vector} of $T$ is defined as $\mu(\rho(T), N^+_\rho)$ where $\rho$
  is any root choice function of $N$.
  We write it as $\mu_r(T, N)$ or
  $\mu_r(T)$ if $N$ is clear from context.
\end{defn}

In Fig.~\ref{fig:compare2nets}, $N$ has one root component $T$
(in brown), with $\mu_r(T,N) = (1{,}1{,}1{,}1{,}1, 3{,}3{,}3)$.
The root component of $N'$ has the same root $\mu$-vector.

\begin{proof}[Proof of \Cref{lem:complementary-mu-vec}]
  The claims obviously hold when $T$ is trivial.
  Now consider $T$ non-trivial and distinct nodes $u \neq v$ in $T$.
  Let $G_{u}$
  (resp.\ $G_{v}$) be a rooted partner of $N$ where $u$ (resp.\ $v$) is a root.
  To prove the first claim, we show that $\mu(u, G_{u}) = \mu(v, G_{v})$ by
  constructing a bijection $f_u$ between $\mathcal{P}_{u}$ and
  $\mathcal{P}_{v}$, where $\mathcal{P}_{z}$ ($z=u,v$) is
  the set of directed paths from $z$ to $x$ in $G_{z}$,
  for an arbitrary but fixed leaf $x$ of $N$.
  Suppose $p_u = u \ldots w \ldots x \in \mathcal{P}_{u}$,
  where $w$ is the last node such that $u \ldots w$ lies in $T$.
  We can modify $p_u$ to a new path $p_v$ by changing the $u \ldots w$ subpath
  to $v \ldots w$, the unique tree path between $v$ and $w$ in $T$.
  By \Cref{undirected before directed},
  the subpath $w \ldots x$ only contain edges in $E_{DP}(N)$.
  Then by \Cref{prop:fixeddirection-in-DP},
  $p_v \in \mathcal{P}_{v}$.
  Obviously, $f_u$ is a bijection whose inverse is
  the map from $\mathcal{P}_{v}$ to $\mathcal{P}_{u}$ constructed by symmetry,
  proving the first claim.
  
  For the second claim, let $u v$ be an edge in $T$ and
  $x$, $G_u$, $G_v$, $\mathcal{P}_u$ and $\mathcal{P}_v$ as before.
  Let $\mathcal{P}_{u}^{v}$ be the set of directed paths from $v$ to $x$ in $G_u$,
  such that $|\mathcal{P}_{u}^{v}|$
  is the coordinate value for $x$ in $\mu_d(u, v)$.
  Define $\mathcal{P}_{v}^{u}$ similarly.
  We can partition $\mathcal{P}_u = A \sqcup B$
  where paths in $A$ contain $v$, and paths in $B$ do not.
  It is easily verified that $B = \mathcal{P}_{v}^{u}$, and that
  prepending $u$ to a path gives a bijection $\mathcal{P}_{u}^{v} \to A$.
  Since $x$ is arbitrary,
  we get $\mu_r(T) = \mu_d(u, v) + \mu_d(v, u)$.
\end{proof}

We are now ready to define the edge-based $\mu$-representation of
an $\mathcal{L}$-network, and an algorithm to compute it.

\begin{defn}[\textbf{edge-based $\mu$-representation}]\label{def:edge-based-mu-rep}
  Let $N$ be a complete $\mathcal{L}$-network. To edge $e$ of $N$
  we associate a set $\mu(e)$, called the \emph{edge $\mu$-vector set}
  of the edge $e$, as follows:
  \begin{itemize}
  \item For $e = (u, v)$, by \Cref{lem:directed-part-edge-and-node}
  we have $v \in V_{DP}$, and we define
  $\mu(e) = \{(\mu(v, N), t)\}$
  using \Cref{def:directed-part-mu-vec}, where
  $t$ is a tag taking value
  $\treetag$ if $e$ is a tree edge, and $\hybtag$ otherwise.
  \item For $e = uv \in E_R$, using \Cref{def:directional-mu-vec} we define
  $\mu(e) = \{(\mu_d(u, v), \treetag), (\mu_d(v, u), \treetag)\}$.
  \end{itemize}
  Let $\mathcal{R}$ be the set of root components of $N$, then the
  \emph{edge-based $\mu$-representation} of $N$, denoted by
  $\mue(N)$, is defined as the multiset
  $$\lbag \mu(e) : e \in E(N) \rbag + \lbag \{(\mu_r(T), \roottag)\}: T \in \mathcal{R}
  \rbag$$
  with $\mu_r$ from \Cref{def:complementary-mu-vec} and $\roottag$ a tag value
  indicating a root $\mu$-vector.
  For an $\mathcal{L}$-network $N'$,
  $\mue(N')$
  is defined as $\mue(\mathcal{C}(N'))$.
\end{defn}

\noindent
In Fig.~\ref{fig:compare2nets} for example,
$\mue(N)$ contains 19 $\mu$-vector sets:
9 in $A_1$, 3 in $A_2$, 6 in $A_3$ and 1 in $A_4$,
using notations as in \Cref{alg:edge-mu-rep}.
$A_1$ contains the unidirectional $\mu$-vector sets from the 8 edges incident
to the leaves, such as
$\{((0{,}0{,}0{,}0{,}0, 1{,}0{,}0), \treetag)\}$ for the edge 
to $h_1$,
and
$\{((0{,}0{,}0{,}0{,}0, 1{,}1{,}0), \treetag)\}$ from $e_5$.
$A_2$ is from the hybrid edges:
$\{((0{,}0{,}0{,}0{,}0, 1{,}1{,}1), \hybtag)\}$ with multiplicity 3.
$A_3$ has only 1 element:
$\{((1{,}1{,}1{,}1{,}1, 3{,}3{,}3), \roottag)\}$.
Finally,
$A_4$ contains the bidirectional $\mu$-vector sets from the 6 edges in the
root component(s). For example, $e_1$ contributes
$\{((1{,}1{,}0{,}0{,}0, 0{,}0{,}0), \treetag), ((0{,}0{,}1{,}1{,}1, 3{,}3{,}3), \treetag)\}$
See the Appendix
for the other 5.

\begin{figure}
  \centering \includegraphics{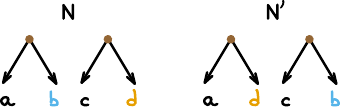}
  \caption{Networks illustrating the need to include the $\mu$-vector
  of trivial root components in \Cref{def:edge-based-mu-rep},
  to distinguish networks with multiple roots.
  $N$ and $N'$ are tree-child and non-isomorphic.
  They have the same multiset of edge $\mu$-vector sets, but different
  root $\mu$-vectors so $\mue(N)\neq \mue(N')$.}
  \label{fig:trivialroot}
\end{figure}

\Cref{lem:complementary-mu-vec} together with \Cref{prop:M-forest} yields
the following
\Cref{alg:edge-mu-rep} to compute the edge-based $\mu$-representation of
an $\mathcal{L}$-network $N = (V, E)$ with $n$ leaves.
As discussed in \Cref{rem:alg-completion},
line~\ref{construct-RP} in \Cref{alg:edge-mu-rep} takes a single traversal of $N$
and $\mathcal{O}(|E|)$ time.  Computing the node-based $\mu$-representation by Algorithm~1
in \crv\ takes $\mathcal{O}(n|E|)$ time.
The remaining steps iterate over edges and take $\mathcal{O}(n|E|)$ time,
giving an overall complexity of $\mathcal{O}(n|E|)$.

\begin{algorithm}
  \caption{Given $\mathcal{L}$-network $N$,
    compute its edge-based
    $\mu$-representation $A = \mue(N)$}
  \label{alg:edge-mu-rep}
  \begin{algorithmic}[1]
    \STATE \label{construct-RP} compute a rooted partner $G$ of $N$, and store:
    \begin{itemize}
      \item $R$ the set of roots in $G$
      \item $\rho$: $V_R(N) \to R$ the function that maps a node
        in a root component $T$ of $N$ to the root of $T$ in $G$
      \item $E_R^+$ the set of edges in $G$ that corresponds to $E_R(N)$
      \item $E_{DP}^+ = E(G) \setminus E_R^+$
     \end{itemize}
     \STATE compute the node-based $\mu$-representation of $G$, \\let $\mu = \muv(\cdot, G)$
     \STATE $A_1 \gets \lbag \{(\mu(v), \treetag)\} : (u, v) \in E_{DP}^+ \cap E_T(G)
    \rbag$
    \STATE $A_2 \gets \lbag \{(\mu(v), \hybtag)\} : (u, v) \in E_{DP}^+ \cap E_H(G) \rbag$
    \STATE $A_3 \gets \lbag \{(\mu(r), \roottag)\} : r \in R \rbag$
    \STATE $A_4 \gets \lbag \{(\mu(v), \treetag), \bigl(\mu(\rho(v)) - \mu(v), \treetag\bigr) \}: (u, v) \in E_R^+ \rbag$
    \RETURN $A = A_1 + A_2 + A_3 + A_4$
  \end{algorithmic}
\end{algorithm}

Compared to the node-based representation $\muv$, $\mue$ has two features.
Unsurprisingly, each undirected edge (whose direction is not resolved by
completion) is represented as bidirectional using two $\mu$-vectors.
This is similar to the representation of edges in unrooted trees,
as bipartitions of $\mathcal{L}$.
The second feature is the inclusion of a $\mu$-vector for each root component,
which may seem surprising.
For a non-trivial root component $T$,
$\mu_r(T)$ is redundant with information from $\mu(e)$ for any $e$ in $T$,
by \Cref{lem:complementary-mu-vec}.
The purpose of including the root $\mu$-vectors in $\mue(N)$
is to keep information from trivial root components, for networks with
multiple roots. Without this information, $\mue$ cannot
discriminate simple networks with multiple roots when one or more
root component is trivial, as illustrated in Fig.~\ref{fig:trivialroot}.

Networks with a unique and non-trivial root component correspond to
standard phylogenetic rooted networks, with uncertainty about the root location.
For these networks, we could use
edge $\mu$-vectors only: $\mue(N) = \lbag \mu(e) : e \in E(N) \rbag$,
that is, omit $A_3$ in \Cref{alg:edge-mu-rep}.
Indeed, the root $\mu$-vector
of the unique root component $T$ is redundant with
$\mu(e)$ of any edge $e$ in $T$.
For this class of standard networks, then, our results below also hold using the
simplified definition of $\mue$.

\subsection{Properties for tree-child networks}
\label{sec:EBMR-treechild}

We will use the following results to reconstruct
a tree-child network from its edge-based $\mu$-representation.  
First we characterize when and how $\mu$-vectors are comparable.

\begin{prop}\label{incomparable components}
  Let $T_1$ and $T_2$ be distinct nontrivial root components of a
  strongly tree-child $\mathcal{L}$-network $N$.
  Then directional $\mu$-vectors from $T_1$
  and from $T_2$ are incomparable to one another.
\end{prop}
\begin{proof}
  Let $uu' \in E(T_1)$ and $vv' \in E(T_2)$.
  Suppose for contradiction that
  $\mu_d(u,u') \geq \mu_d(v,v')$.
  Let $G$ be a rooted partner of $N$ in which $u$ and $v$ are roots.
  Then $\mu_d(u,u') = \mu(u',G)$ and $\mu_d(v,v') = \mu(v',G)$.
  Since $G$ is tree-child, there exists a path $u' \leadsto v'$
  in $G$ by \Cref{lem:c9-5}
  (possibly up to relabeling if $\mu_d(u,u') = \mu_d(v,v')$) 
  Since $(v,v')$ is a tree edge in $G$, by Lemma~1 in \crv,
  $u' \leadsto v'$ contains or is contained in $(v,v')$.
  Both cases imply that $u'\in\{v,v'\}$
  (using that $v$ is a root for the first case),
  a contradiction. 
\end{proof}

\begin{prop}\label{prop:incomparable-root-mu-vec}
  In a weakly tree-child $\mathcal{L}$-network,
  different root components have incomparable root $\mu$-vectors.
\end{prop}
\begin{proof}
  Let $T_1\neq T_2$ be
  root components of a tree-child $\mathcal{L}$-network $N$.
  In a tree-child rooted partner of $N$,
  there is no directed path between the roots of $T_1$ and $T_2$.
  Therefore, by Lemmas~\ref{lem:c9-5} and~\ref{lem:complementary-mu-vec},
  $\mu_r(T_1)$ and $\mu_r(T_2)$ are incomparable.
\end{proof}

\begin{lem}\label{lem:up-down-stream}
  Let $N$ be a strongly tree-child $\mathcal{L}$-network.
  Suppose $uv, st$ are two
  (not necessarily distinct) edges 
  in root component $T$ of $N$ such that the undirected tree path from $u$
  to $t$ in $T$ contains $v$ and $s$.  Then:
  \begin{enumerate}
  \item $\mu_d(u, v) \geq \mu_d(s, t)$,
  \item $\mu_d(v, u)$ is incomparable to $\mu_d(s, t)$,
  \item $\mu_d(u, v)$ is incomparable to $\mu_d(t, s)$.
  \end{enumerate}
\end{lem}

\begin{proof}
  Let $G_u$ (resp. $G_v$) be a rooted partner of $N$ with $u$ (resp. $v$) as a root.
  For part 1, by \Cref{lem:c9-5}, we have
  $\mu_d(u, v) = \mu(v, G_u) \geq \mu(t, G_u) = \mu_d(s, t)$.

  For part 2, by symmetry, it suffices to show that
  $\mu_d(v, u) \not\leq \mu_d(s, t)$.
  Let $w_1, \ldots, w_k$ be the neighbors of $u$ besides $v$.
  Then $\mu(w_i, G_u) = \mu(w_i, G_v)$
  by \Cref{prop:directed-part-mu-vec} if $w_i \in V_{DP}(N)$, or
  \Cref{prop:directional-mu-vec} if $w_i \in V(T)$.
  Then by \Cref{lem:c9-4} we have
  $\mu_d(v, u) = \mu(u, G_v) = \sum_{i=1}^k \sum_{(u,w_i) \in E(G_v)}
  \mu(w_i,G_v) = \sum_{i=1}^k \sum_{(u,w_i) \in E(G_v)} \mu(w_i,G_u)$.

  First, suppose for contradiction that $\mu_d(v, u) < \mu_d(s, t) = \mu(t, G_u)$.
  Then for each $i$, $\mu(t,G_u) > \mu(w_i,G_u)$,
  and since $G_u$ is tree-child
  there exists a path $t \leadsto w_i$ in $G_u$ by \Cref{lem:c9-5}.
  As $u$ is a root in $G_u$
  and not contained in these paths, $w_1,\ldots,w_k$ are hybrid nodes.  Then $u$
  does not have a tree child in $G_v$, a contradiction.

  Now suppose instead $\mu_d(v, u) = \mu_d(s, t) = \mu(t, G_u)$, then
  $\mu(t,G_u) \geq \mu(w_i,G_u)$ for all $i$.
  If $\mu(t,G_u) = \mu(w_i,G_u)$ for some $i$, then $w_i = w_1$ is the only
  neighbor of $u$ other than $v$.
  By \Cref{lem:c9-5}, $t$ and $w_1$ are connected by an elementary path in $G_u$,
  which is impossible as both have $u$ as a parent.
  Therefore $\mu(t,G_u) > \mu(w_i,G_u)$ for all $i$, which
  leads to a contradiction by the argument above.

  Part 3 follows from part 2 using that
  $\mu_d(a, b) + \mu_d(b, a) = \mu_r(T)$ for any $ab \in E(T)$
  by \Cref{lem:complementary-mu-vec}.
\end{proof}

Next, we relate edges with identical directional $\mu$-vectors.

\begin{lem}\label{lem:elem-path}
  Let $N$ be a strongly tree-child $\mathcal{L}$-network,
  and $x$ a fixed $\mu$-vector.
  If we direct all edges $uv \in E_R(N)$
  with $\mu_d(u, v) = x$ as
  $(u, v)$, then these edges form a directed path.
  If the path is nonempty, we denote the first node as $h(x, N)$.
\end{lem}

\begin{proof}
  Let $E_x = \{uv \in E_R(N): \mu_d(u, v) = x\}$.
  Take $uv$ and $st$ in $E_x$.
  By \Cref{incomparable components}, they
  are in the same root component $T$,
  so there is an undirected path $p$ in $T$
  connecting $uv$ and $st$.
  By applying \Cref{lem:up-down-stream} multiple times
  and permuting 
  labels if necessary, we may assume $p$
  is of the form $uv \ldots st$ with $\mu_d(u, v) = \mu_d(s, t) = x$.
  For a rooted partner $G$ of $N$ with $u$ a root, we have
  $\mu(v, G) = \mu(t, G) = x$, which implies by \Cref{lem:c9-5}
  there is an elementary path in
  $G$ from $v$ to $t$.  By \Cref{undirected before directed}, this path lies
  in $E_R(N)$. But since $E_R(N)$ induces a forest, this
  elementary path must be the $v \ldots st$ part of the path $p$. Therefore
  all the intermediate nodes $w$ in $p$ have $\undegree(w, N) = 2$ and
  $\indegree(w, N) = \outdegree(w, N) = 0$.  Furthermore, if $w_1$ and $w_2$ are
  consecutive nodes in $p$, then $w_1w_2\in E_x$ because
  $x=\mu(v, G) \geq \mu(w_2, G) = \mu_d(w_1, w_2) \geq \mu(t, G)=x$.
  Therefore all edges in $p$ are in $E_x$, and form a directed path
  when directed as in the statement.
  
  Now take an undirected path $p_0$ of edges in $E_x$, of maximum length, and let
  $e$ one of its edges.
  To show that $p_0$ contains all the edges in $E_x$, take $e'\in E_x$.
  By the previous argument, $e$ and $e'$ are connected by a tree path $p_1$.
  Also by the previous argument, all intermediate nodes in $p_0$ and in $p_1$
  have $\undegree = 2$. Since $p_1$ cannot extend $p_0$ by definition of $p_0$,
  $p_1$ must be contained in $p_0$, therefore $e'$ is in $p_0$ as claimed.
\end{proof}

Finally, the next result was proved for orchard DAGs,
which include tree-child DAGs
\citep[Proposition~10]{2024Cardona-extendedmu}.
We restrict its statement to hybrid nodes here, because
we allow networks to have in and out degree-1 nodes.
\begin{lem}\label{unique hybrid mu-vector}
  Let $G$ be a tree-child $\mathcal{L}$-DAG. Let $u,v$ be
  distinct hybrid nodes in $G$. Then
  $\mu(u,G) \neq \mu(v,G)$.
\end{lem}

\subsection{Reconstructing a complete tree-child network}\label{sec:reconstruct}

To reconstruct a complete tree-child network $N$ from its
edge-based $\mu$-representation,
\Cref{alg:canon-dag} will first construct $\muv(G)$
for a rooted partner $G$ from $\mue(N)$.
Then \Cref{alg:network-from-rooted-partner} will use
$\mue(N)$ to undirect some edges in $G$ and recover $N$.

\begin{algorithm}[H]
  \caption{Given $A = \mue(N)$
    from a
    tree-child $\mathcal{L}$-network $N$, compute
    $B = \muv(G)$ for some rooted partner $G$ of $N$
    }
  \label{alg:canon-dag}
  \begin{algorithmic}[1]
    \REQUIRE multiset $A$
    \ENSURE multiset $B$
    \STATE $B_1 \leftarrow \lbag x \colon \{(x,\treetag)\} \in A\rbag$
    \STATE $B_2 \leftarrow \{ x \colon \{(x,\hybtag)\} \in A \}$
    \STATE $B_3 \leftarrow \{ x \colon \{(x,\roottag)\} \in A \}$
    \STATE $B_4 \gets \lbag \rbag$
    \FOR{$z$ in $B_3$} \label{ln:B3loop}
      \STATE $M(z) \leftarrow \lbag x : \{(x,\treetag),(z-x,\treetag)\} \in A\rbag$
      \label{ln:Mz}
      \STATE \textbf{if} $M(z)$ is empty \textbf{then}
      skip to next iteration
      \STATE {$r(z) \gets$ some arbitrary element of $M(z)$} \label{ln:rs}
      \FOR{$\{(x_1,\treetag),(x_2,\treetag)\} \in A$ with $x_1+x_2=z$} \label{ln:pick-from-pair}
      \STATE $B_4 \gets B_4 +\lbag y \rbag$
      where $y=x_i$ if $x_i\leq r(z)$
      else $y=z-x_i$ if $x_i>r(z)$ ($i=1$ or $2$)
      \label{ln:pick-from-pair-choice}
      \ENDFOR
    \ENDFOR
    \STATE $B \leftarrow B_1 + B_2 + B_3 + B_4$
    \RETURN $B$
  \end{algorithmic}
\end{algorithm}

Continuing with $N$ in Fig.~\ref{fig:compare2nets} (left),
$A=\mue(N)$ is given in the Appendix.
\Cref{alg:canon-dag} starts with
$B_1 = \{
  (1{,}0{,}0{,}0{,}0, 0{,}0{,}0),
  \ldots,
  (0{,}0{,}0{,}0{,}0, 0{,}0{,}1),
  (0{,}0{,}0{,}0{,}0, 1{,}1{,}0)\}$
for edges incident to leaves and $e_5$ (in black in Fig.~\ref{fig:compare2nets}).
$B_2 = \{(0{,}0{,}0{,}0{,}0, 1{,}1{,}1)\}$,
for the unique $\mu$-vector shared by all 3 hybrid edges in $N$.
$B_3$ has a single element $z = (1{,}1{,}1{,}1{,}1, 3{,}3{,}3)$
because $N$ has a single root component,
so the loop on line~\ref{ln:B3loop} has a single iteration and
all elements $\{(x_1,\treetag),(x_2,\treetag)\}$ in $A$ satisfy $x_1+x_2=z$
(from edges in brown in Fig.~\ref{fig:compare2nets}).
On line~\ref{ln:Mz}, $M(z)$ has 12 elements
(see the Appendix).
We can arbitrarily
pick $r(z)= (0{,}0{,}0{,}1{,}1, 2{,}2{,}2) \in M(z)$, which corresponds
to $e_7$ directed rightward. Then
$B=\muv(G)$ for the partner $G$ of $N$ rooted at the
node incident to $e_6$ and $e_7$ (see Fig.~\ref{fig:compare2nets-G}).

\begin{proof}[Proof of correctness for \Cref{alg:canon-dag}]
  As $N$ and $\mathcal{C}(N)$ have the same
  rooted partners and $\mue(N) =  \mue(\mathcal{C}(N))$, we may assume $N$
  to be complete.

  Let $\rho$ be a root choice function such that $\rho(T) = h(r(\mu_r(T)), N)$,
  where $r$ is the function on line~\ref{ln:rs} and $h$ is defined in
  \Cref{lem:elem-path}.  By \Cref{lem:complementary-mu-vec} and
  \Cref{prop:incomparable-root-mu-vec}, $\rho(T) \in V(T)$ is
  well-defined.
  Let $G = N_\rho^+$. We shall show that \Cref{alg:canon-dag}, with
  $A = \mue(N)$ as input, produces the output $B = \muv(G)$.

  \noindent
  Consider partitioning $V(N) = V(G)$ into the following sets:
  \begin{enumerate}
  \item[] $V_1 = \{v$ is a tree node in the directed part of $N\}$,
  \item[] $V_2 = \{v$ is a hybrid node$\}$,
  \item[] $V_3 = \{v$ is a root in $G\}$,
  \item[] $V_4 = \{v$ in a root component of $N$, but not a root in $G\}$.
  \end{enumerate}
  We will establish
  $B_i = \lbag \mu(v, G): v \in V_i \rbag$ for the
  multisets $B_i$ in the algorithm ($i=1,\ldots,4$), to conclude
  the proof.

  By \Cref{lem:directed-part-edge-and-node},
  $(u, v) \mapsto v$ is a bijection between
  the directed tree edges and $V_1$.
  Then by \Cref{def:edge-based-mu-rep},
  $B_1 = \lbag \mu(v, N): (u, v) \in E_T(N) \rbag = \lbag \mu(v, G): v \in V_1 \rbag$,
  which concludes case $i=1$.

  For $i = 2$, \Cref{unique hybrid mu-vector} implies that
  $\mu(u, G) \neq \mu(v, G)$ for distinct $u \neq v$ in $V_2$.
  Therefore
  $\lbag \mu(v, G): v \in V_2 \rbag = \{ \mu(v, G):
  v \in V_2 \}$.  By the definition of hybrid edges,
  $B_2 = \{x \colon \{(x,\hybtag)\} \in A \} = \{\mu(v, G):
  (u, v) \in E_H(N)\}$ is equal to
  $\{\mu(v, G): v \in V_2 \}$, which implies
  $B_2 = \lbag \mu(v, G): v \in V_2 \rbag$.

  For $i = 3$, by \Cref{lem:complementary-mu-vec} and
  \Cref{prop:incomparable-root-mu-vec}, the $\mu$-vectors of the roots
  in $G$ are the same as the root $\mu$-vectors, and
  are all distinct.  Hence
  $B_3 = \lbag \mu(v, G): v \in V_3 \rbag$.

  For $i = 4$, let $E_R^+$ be the set of edges in $G$ that
  corresponds to $E_R(N)$. Consider the map $V_4 \to E_R^+$ that
  associates $v$ to its parent edge $(u,v)$ in $G$.
  It is well-defined because $V_4$ excludes the roots of $G$,
  root components only contain tree nodes (\Cref{prop:M-forest}) and
  $uv\in E_R(N)$ by \Cref{lem:directed-part-edge-and-node}.
  Furthermore, the map is a bijection.  Therefore we have
  $\lbag \mu(v, G) : v \in V_4 \rbag = \lbag \mu_d(u, v) : (u, v) \in E_R^+
  \rbag$.

  $B_4$ is constructed in \Cref{alg:canon-dag}
  by taking a $\mu$-vector from the pair
  $\mu_d(s, t)$ and $\mu_d(t, s)$, for each undirected edge $st$ in each
  root component.
  Let $T$ be the root component that contains $st$ and let
  $z = \mu_r(T) \in B_3$.
  Then $\mu_d(s, t) + \mu_d(t, s)$ equals $z$ but no other root $\mu$-vector
  by \Cref{prop:incomparable-root-mu-vec}, so $\mu(st)$ is considered
  at exactly one iteration of the loop on line~\ref{ln:pick-from-pair}.
  Next we need to show that on line~\ref{ln:pick-from-pair-choice},
  exactly one $\mu$-vector gets chosen, and is
  $y = \mu_d(s, t)$ where $(s, t) \in E_R^+$.

  From \Cref{lem:elem-path},
  let $u = h(r(z), N)$ be the root of $T$ in $G$
  and $v$ such that $r(z) = \mu_d(u, v)$.
  Since $u$ is a root in $G$ and $(s, t) \in E(G)$,
  the tree path $p$ in $T$ from $u$ to $t$ contains $s$.
  If $p$ also contains $v$, then
  $\mu_d(s, t) \leq \mu_d(u, v)$ and $\mu_d(t, s) \not\leq \mu_d(u, v)$
  by \Cref{lem:up-down-stream}, so
  line~\ref{ln:pick-from-pair-choice} defines $y = \mu_d(s, t)$ as claimed.
  If $p$ does not contain $v$, then the tree path from $v$ to $t$ contains
  $u$ and $s$, so by \Cref{lem:up-down-stream}
  $\mu_d(s, t)$ is incomparable to $\mu_d(u, v)$ and
  $\mu_d(t, s) \geq \mu_d(u, v)$.
  Further, $\mu_d(t, s) > \mu_d(u, v)$ by the choice $u = h(r(z), N)$
  and \Cref{lem:elem-path}.
  Therefore line~\ref{ln:pick-from-pair-choice} defines
  $y=z-\mu_d(t, s) = \mu_d(s, t)$,
  which concludes the proof.
\end{proof}

\begin{algorithm}[H]
  \caption{Given $A = \mue(N)$ from a
    tree-child $\mathcal{L}$-network $N$, and a rooted partner
    $G$ of $N$, modify $G$ to obtain $\mathcal{C}(N)$ }
  \label{alg:network-from-rooted-partner}
  \begin{algorithmic}[1]
    \STATE $B \leftarrow \mue(G)$
    \STATE $F \leftarrow \lbag x : \{(x,\treetag)\} \in B - A \rbag$ \label{ln:M-nodes}
    \FOR{$x \in \text{Unique}(F)$}
      \STATE $m(x) \leftarrow $ multiplicity of $x$ in $F$
      \STATE $p(x) \leftarrow $ the directed path in $G$ formed by
      $\{(u,v)\in E(G): \mu_V(v, G) = x\}$
      \label{ln:elementary-path}
      \STATE undirect the first $m(x)$ edges in $p(x)$ \label{ln:undirect}
    \ENDFOR
    \RETURN $G$
  \end{algorithmic}
\end{algorithm}

Given $\mue(N)$ from $N$ in Fig.~\ref{fig:compare2nets} and $G$
from \Cref{alg:canon-dag} (Fig.~\ref{fig:compare2nets-G}),
$F$ contains 6 $\mu$-vectors
(see the Appendix)
so 6 edges in $G$
are undirected to obtain $N$.
One of them $x=(0{,}0{,}1{,}0{,}0, 0{,}0{,}0)$ has multiplicity 1 in $F$
but corresponds to an elementary path of 2 edges in $G$
adjacent to $b$. On this path, only $e_2$ is undirected by
\Cref{alg:network-from-rooted-partner}.

\begin{proof}[Proof of correctness of \Cref{alg:network-from-rooted-partner}]
  Note that line~\ref{ln:elementary-path} uses \Cref{lem:c9-5} to claim that
  $p(x)$ is a directed path,
  and so line~\ref{ln:undirect} can be applied.

  Let $E_R^+$ denote the set of edges in $G$ that corresponds to
  edges in $E_R(N)$.  It suffices to show that line~\ref{ln:undirect}
  undirects all edges in $E_R^+$ and no other.
  Obviously, line~\ref{ln:M-nodes} defines
  $F = \lbag \mu(v, G): (u, v) \in E_R^+ \rbag$.
  Suppose $F$ consists of elements $x_1, \ldots, x_k$ with
  multiplicities $m_1, \ldots, m_k$. 
  We know that $E_R^+$ exactly consists of $m_i$ edges whose children
  have $\mu$-vector $x_i$, for $i = 1, \ldots, k$.  Thus we only need
  to show that for each $x_i$, 
  the first $m_i$ edges in $p(x_i)$
  are in $E_R^+$.
  By \Cref{undirected before directed}, if an edge $e$
  is not in $E_R^+$, then all edges below it are also not in $E_R^+$.
  Therefore along the path $p(x_i)$, edges in $E_R^+$ must come first
  before any edge not in $E_R^+$,
  which finishes the proof.
\end{proof}

The following theorem derives directly from Theorem~1 in \crv\ to
reconstruct $G$ from $\muv(G)$, and the application of
Algorithms~\ref{alg:canon-dag} and~\ref{alg:network-from-rooted-partner}.

\begin{thm}\label{thm:samemu-samenet}
  Let $N_1$ and $N_2$ be strongly tree-child $\mathcal{L}$-networks.
  Then $\mue(N_1) = \mue(N_2)$ if and
  only if $N_1$ and $N_2$ are phylogenetically isomorphic.
\end{thm}

\Cref{thm:samemu-samenet} does not generally hold for weakly tree-child
networks, as seen in a counter-example in Fig.~\ref{fig:thmcounterex}.

\begin{figure}
\centering
\centerline{\includegraphics{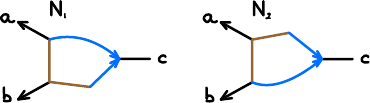}}
\caption{Weakly tree-child $(a,b,c)$-networks
  for which
  \Cref{thm:samemu-samenet} does not hold.
  For each network, the rooted partner rooted at the degree-2 node is tree-child.
  The other 2 rooted partners are not.
  $N_1$ and $N_2$ are not phylogenetically isomorphic
  yet
  $\mue(N_1) = \mue(N_2) = A_1 + \cdots + A_4$ with
  $A_1=\lbag
  \{((1,0,0), \treetag)\},
  \{((0,1,0), \treetag)\},
  \{((0,0,1), \treetag)\}\rbag$,
  \\
  $A_2=\lbag
  \{((0,0,1), \hybtag)\},
  \{((0,0,1), \hybtag)\}\rbag$,
  $A_3=\lbag
  \{((1,1,2), \roottag)\}\rbag$ and
  $A_4=\lbag
  \{((1,0,1), \treetag), ((0,1,1), \treetag)\},
  \{((1,1,1), \treetag), ((0,0,1), \treetag)\}
  \rbag$.
}\label{fig:thmcounterex}
\end{figure}

\section{The edge-based $\mu$-distance}

\begin{defn}[\textbf{edge-based $\mu$-distance}]\label{metric}
  Let $N_1$ and $N_2$ be $\mathcal{L}$-networks.
  The \emph{edge-based} $\mu$-dissimilarity between
  $N_1$ and $N_2$ is defined as
  $$\dmue(N_1,N_2) = |\mue(N_1) \triangle \mue(N_2)|\;.$$
\end{defn}

\noindent
For example, $\dmue(N,N') = 2$ in Fig.~\ref{fig:nodebased-vs-edgebased-distance}.
For the networks in Fig.~\ref{fig:compare2nets}, $\dmue(N,N')=5$
due to non-matching $\mu$-vector sets for edges $e_6$ and $e_7$ in $N$
and the 3 unlabelled tree edges in $N'$.
(see Fig.~\ref{fig:compare2nets} and the Appendix
for details).

We are now ready to state our main theorem,
which justifies why we may refer to $\dmue$ as a distance.

\begin{figure}
  \centerline{\includegraphics{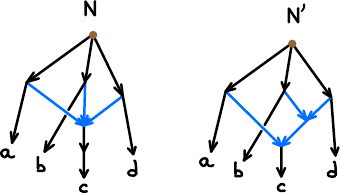}}
  \caption{Example rooted networks for which the node-based
  and edge-based $\mu$ distances differ, on leaves $(a,b,c,d)$:
  Left: $N$ is tree-child, non-bicombining.
  Right: $N'$ is not tree-child, but bicombining.
  Both have 3 nodes (including leaf $c$) with $\mu$-vector $(0{,}0{,}1{,}0)$,
  and $\dmuv(N,N')=0$.
  $N$ and $N'$ both have 5 edges with $\mu$-vector $(0{,}0{,}1{,}0)$,
  of which 3 (resp. 4) are of hybrid type in $N$ (resp. $N'$),
  and $\dmue(N,N')=2$.
  }\label{fig:nodebased-vs-edgebased-distance}
\end{figure}

\begin{thm}\label{metric correctness}
  For a vector of leaf labels $\mathcal L$, $\dmue$ is a distance on the class of
  (complete) strongly tree-child $\mathcal L$-networks.
\end{thm}
\begin{proof}
  From the properties of the symmetric difference,
  $\dmue$ is a dissimilarity
  in the sense that is it symmetric, non-negative, and satisfies the
  triangle inequality. It remains to show that $\dmue$ satisfies
  the separation property.
  Let $N_1$ and $N_2$ be tree-child $\mathcal{L}$-networks.
  If $\dmue(N_1,N_2) = 0$ then $\mue(N_1) = \mue(N_2)$ and
  by \Cref{thm:samemu-samenet}, $N_1 \cong N_2$.
\end{proof}

For unrooted trees, the $\mu$-vector of each undirected edge encodes
the bipartition on $\mathcal{L}$ associated with the edge,
hence $\dmue$ agrees with the Robinson-Foulds distance on
unrooted trees.

On rooted trees, $\dmue$ agrees with $\dmuv$. Indeed, if $T$ is a directed tree
or forest on $\mathcal{L}$,
then each non-root node $v$ has a unique
parent edge $e$ with element $\{(\muv(v),\treetag)\}$ in $\mue(T)$;
and each root $u$ forms to a trivial root component with element
$\mue(u) = \{(\muv(u),\roottag)\}$ in $\mue(T)$.

However, $\dmue$ does not generally extend $\dmuv$.
For example, consider the rooted networks in Fig.~\ref{fig:nodebased-vs-edgebased-distance}.
They have the same $\muv$ representation, hence $\dmuv(N,N') = 0$.
However, their $\mue$ representations differ, due
to edges with the same $\mu$ vector (1 path to $c$ only) but different tags
(tree edge in $N$ versus hybrid edge in $N'$).
Hence $\dmue$ can distinguish these networks: $\dmue(N,N') > 0$.

\bigskip

We can compute $\dmue$ using a variant of Algorithm~3 in \crv.
Specifically, we first group the elements of $\mue(N_i)$ ($i=1,2$)
by their type: of the form $\{(x, \roottag)\}$, $\{(x, \treetag)\}$,
$\{(x, \hybtag)\}$, or $\{(x, \treetag), (y, \treetag)\}$.
Then it suffices to equip a total order and apply Algorithm~3 in \crv\
to each group, then add the distances obtained from the 4 groups.
For the first 3 types we can simply use the lexical order on the $\mu$-vector
$x$.
For the last type, we may compare two elements by comparing the
lexically smaller $\mu$-vector first and then the larger one, to obtain a
total order within the group.

As in \crv, with $\mue(N_i)$ computed and sorted, the above
takes $\mathcal{O}(n|E|)$ time where $|E| = \max(E_1, E_2)$.  Taking into account computing
and sorting $\mue(N_i)$, computing $\dmue$ takes $O\bigl(|E|(n + \log|E|)\bigr)$ time.

This complexity can also be expressed in terms of the number of leaves and root
components, thanks to the following straightforward generalization of
Proposition~1 in \crv, allowing for multiple root components.

\begin{prop}\label{prop:SDAG-bound}
  Let $N$ be a tree-child $\mathcal{L}$-network with $n$ leaves and $t$ root components.
  Then
  $|V_H| \leq n - t$.\\
  A node $v$ is called \emph{elementary}
  if it is a tree node and either
  $\outdegree(v)=\degree(v)=1$, 
  or $\outdegree(v)<\degree(v)=2$.
  If $N$ has no elementary nodes
  then
  \[
    |V| \leq 2n - t + \sum_{v \in V_H} \indegree(v) \leq (m+2)(n - t) + t
  \]
  where $\displaystyle m = \max_{v \in V_H}\{\indegree(v)\}$,
  and $|E| \leq (2m+1)(n-t)$.
\end{prop}

\begin{proof}
  Since $N$ is an $\mathcal{L}$-network, it has no ambiguous leaves,
  and the elementary nodes in $N$
  are the tree nodes of out-degree 1 in any rooted partner.
  By considering a rooted partner, we may assume that $N$ is a DAG with $t$ roots,
  and follow the proof of Proposition~1 in \crv.
  Their arguments remain valid for the bounds on $|V_H|$ and $|V|$ when
  $N$ has $t\geq 1$ roots,
  and when the removal of all but 1 parent hybrid edges at each hybrid node
  gives a forest instead of a tree.
  To bound $|E|$ we enumerate the parent edges of each node:
  $|E| \leq (|V_T| - t) + m |V_H| = |V| - t + (m-1)|V_H|$
  then use the previous bounds.
\end{proof}

Therefore, as long as $m$ and $t$ are bounded
and there are no elementary nodes, for example in binary
tree-child networks with a single root component, then $|E| = \mathcal{O}(n)$.
Consequently, computing $\mue$ on one such network or computing $\dmue$ on
two such networks takes $\mathcal{O}(n^2)$ time.  

\section{Conclusion and extensions}

For rooted networks, the node-based representation $\muv$, or equivalently
the ancestral profile, is known to provide a distance between
networks beyond the class of tree-child networks,
such as the class of semibinary tree-sibling time-consistent networks
\citep{2008Cardona-mudist-treesibling}
and stack-free orchard binary networks \citep{2021Bai-ancestralprofile},
a class that includes binary tree-child networks.
Orchard networks can be characterized as rooted trees with
additional ``horizontal arcs'' \citep{2022vanIerselJanssenJonesMurakami}.
They were first defined as cherry-picking networks: networks
that can be reduced to a single edge
by iteratively reducing a cherry or a reticulated cherry
\citep{2021JanssenMurakami}.
A \emph{cherry} is a pair of leaves $(x,y)$ with a common parent.
A \emph{reticulated cherry} is a pair of leaves $(x,y)$ such that
the parent $u$ of $y$ is a tree node and
the parent $v$ of $x$ is a hybrid node with $e=(u,v)$ as a parent hybrid edge.
Reducing the pair $C=(x,y)$ means removing taxon $x$ if $C$ is a cherry or
removing hybrid edge $e$ if $C$ is a reticulated cherry,
and subsequently suppressing $u$ and $v$ if they are of degree 2.
Cherries and reticulated cherries are both well-defined
on the class of semidirected networks considered here,
because leaves are well-defined (stable across rooted partners),
each leaf is incident to a single tree edge,
and hybrid nodes / edges are well-defined.
A \emph{stack} is a pair of hybrid nodes connected by a hybrid edge,
and a rooted network is \emph{stack-free} if it has no stack.
As hybrid edges are well-defined on our general class of networks, the concepts
of stacks and stack-free networks also generalize directly.
Therefore, we conjecture that for semidirected networks, our edge-based
representation $\mue$ and the associated dissimilarity $\dmue$ also separate
distinct networks well beyond the tree-child class, possibly to
stack-free orchard semidirected networks.

\medskip
To discriminate distinct orchard networks with possible stacks,
\citet{2024Cardona-extendedmu} introduced an ``extended'' node-based
$\mu$-representation of rooted phylogenetic networks.
In this representation, the $\mu$-vector for each node $v$ is extended by
one more coordinate, $\mu_0(v)$, counting the number of paths from $v$
to a hybrid node (any hybrid node).
On rooted networks, adding this extension allows $\muv$ to distinguish between
any two orchard networks, even if they contain stacks
(but assumed binary, without parallel edges and without outdegree-1 tree nodes
in \cite{2024Cardona-extendedmu}).
For semidirected networks, we conjecture that the edge-based representation
$\mue$ can also be extended in the same way, and that this extension may
provide a proper distance on the space of semidirected orchard networks.

\medskip

Phylogenetic networks are most often used as metric networks
with edge lengths and inheritance probabilities.
Dissimilarities are needed to compare metric networks using both
their topologies and edge parameters.
For trees, extensions of the RF distance, which $\dmue$ extends,
are widely used. They can be expressed using edge-based $\mu$-vectors as
\begin{equation}\label{eq:weightedRF}
  d(T_1,T_2) = \sum_{m\in\mu_E(T_1)\cup \mu_E(T_2)} | \ell(m,T_1) - \ell(m,T_2) |^p
\end{equation}
where $\ell(m,T_i)$ is the length in tree $T_i$ of the edge corresponding
to the $\mu$-vector $m$,
considered to be $0$ if $m$ is absent from $\mu_E(T_i)$.
The weighted RF distance uses $p=1$ \citep{1979RobinsonFoulds}
and the branch score distance uses $p=2$
\citep{1994KuhnerFelsenstein-branchscoredistance}.
If all weights $\ell(m,T)$ are 1 for $m\in \mu_E(T)$,
then \eqref{eq:weightedRF} boils down to the RF distance
when restricted to trees (either rooted or unrooted),
and to our $\dmue$ dissimilarity
on semidirected phylogenetic networks more generally.
For networks with edge lengths, \eqref{eq:weightedRF}
could be used to extend $\dmue$, where $\ell(\mue(e),N)$ is defined
as the length of edge $e$ in $N$ as it is for trees.
A root $\mu$-vector could be assigned weight
$\ell(\mu_r(T),N)=0$, because in standard cases,
such as for networks with a single root component, the root $\mu$-vector(s)
carry redundant information.

Alternatively, using inheritance probabilities
could be
useful to capture the similarity between a network having a hybrid
edge with inheritance very close to $0$ and a network lacking this edge.
To this end, we could modify
$\mu$-vectors.
Recall that \crv\ defined $\mu(v,N)=(\mu_1,\ldots,\mu_n)$ with $\mu_i$ equal to
the number $m_i$ of paths from $v$ to taxon $i$ in a directed network $N$.
We could generalize $\mu_i$ to be a function of these $m_i$ paths,
possibly reflecting inheritance probabilities.
For example, we could use
the weight of a path $p$, defined as $\gamma(p) = \prod_{e\in p} \gamma(e)$.
These weights sum to 1
over up-down paths between $v$ and $i$ \citep{2023XuAne},
although not over the $m_i$ directed paths from $v$ to $i$.
The weights of the $m_i$ paths could then be normalized before calculating
their entropy $H_i = - \sum_{p: v\leadsto i} \gamma(p) \log\gamma(p)$
and then define $\mu_i = e^{H_i}$.
The original definition $\mu_i = m_i$ corresponds
to giving all paths $v\leadsto i$ equal weight $1/m_i$.
This extension carries over from directed to
semidirected networks because we proved here that the set of directed paths
from an edge $e=(u,v)$ to $i$ is independent of the root choice,
given a fixed admissible direction assigned to $e$, as shown in
Propositions~\ref{prop:directed-part-mu-vec} and~\ref{prop:directional-mu-vec}.
With this extension, $\mu$-vectors are in the continuous space
$\mathbb{R}_{\geq 0}^n$ instead of $\mathbb{Z}_{\geq 0}^n$.
To use them in a dissimilarity between networks $N$ and $N'$, we could use
non-trivial distance between $\mu$-vectors (such as the $L^1$ or $L^2$ norm)
then get the score of an optimal matching between $\mu$-vectors in $\mu_E(N)$ and
$\mu_E(N')$. Searching for an optimal matching would increase the computational
complexity of the dissimilarity, but would remain polynomial
using the Hungarian algorithm \citep{1955Kuhn-hungarian}.

\medskip

To reduce the dependence of $\dmue$
on the number of taxa $n$ in the two networks,
$\dmue$ should be normalized by a factor depending on $n$ only.
This is particularly useful to compare networks with different leaf labels,
by taking the dissimilarity between the subnetworks on their shared leaves.
Ideally, the normalization factor is the diameter of the network space, that is,
the maximum distance $\dmue(N,N')$ over all networks $N$ and $N'$ in a subspace
of interest.
For the subspace of unrooted trees on $n$ leaves, this is $2(n-3)$
\citep{steel16_phylogeny}.
Future work could study the diameter of other semidirected network spaces,
such as level-1 or tree-child semidirected networks
(which have $n-t$ or fewer hybrid nodes, \Cref{prop:SDAG-bound})
or orchard semidirected networks (whose number of hybrids is unbounded).

To compare semidirected networks $N_1$ on leaf set $\mathcal{L}_1$
and $N_2$ on leaf set $\mathcal{L}_2$ with a non-zero dissimilarity if
$\mathcal{L}_1 \neq \mathcal{L}_2$, one idea is to consider the subnetworks
$\tilde{N}_1$ and $\tilde{N}_2$ on their common leaf set
$\mathcal{L} = \mathcal{L}_1 \cap \mathcal{L}_2$ then
use a penalized dissimilarity:
$$\dmue(\tilde{N}_1,\tilde{N}_2) + \lambda d_\mathrm{Symm}(\mathcal{L}_1,\mathcal{L}_2)$$
for some constant $\lambda \geq 0$.
This dissimilarity may not satisfy the triangle inequality,
which might be acceptable in some contexts.
For example, consider as input a set of semidirected networks $N_1,\ldots, N_n$
with $N_i$ on leaf set $\mathcal{L}_i$, and consider
the full leaf set $\mathcal{L} =\cup_i^n \mathcal{L}_i$.
We may then seek an $\mathcal{L}$-network $N$
that minimizes some criterion, such as
\begin{equation}\label{eq:supertree}
\sum_{i=1}^n d(N,N_i)\;.
\end{equation}

When $N$ is constrained to be an unrooted tree,
input networks $N_i$ are unrooted trees and when
$d$ is the RF distance using $N$ pruned to $\mathcal{L}_i$,
this is the well-studied RF supertree problem
\citep{2017VachaspatiWarnow-RFsupertree}.
When the input trees $N_i$ are further restricted to be on 4 taxa,
\eqref{eq:supertree} is the criterion used by ASTRAL \citep{2018Zhang-ASTRAL3}.
The very wide use of ASTRAL and its high accuracy points to the impact of
distances that are fast to calculate, such as our proposed $\dmue$.

\section*{Acknowledgements}
\noindent
This work was supported in part by the National Science Foundation
(DMS 2023239) and by a H. I. Romnes faculty fellowship
to C.A. provided by the University of Wisconsin-Madison Office of the
Vice Chancellor for Research with funding from the
Wisconsin Alumni Research Foundation.

\bibliography{lib}

\com{
\vspace{-33pt}
\begin{IEEEbiographynophoto}{Michael Maxfield}
received his Bachelor of Science degrees in Computer Science and Mathematics from the
University of Wisconsin - Madison in 2024, and is currently working towards his
Masters degree in Mathematics. His interests are in various areas of mathematics and logic.
\end{IEEEbiographynophoto}

\vspace{-33pt}
\begin{IEEEbiographynophoto}{Jingcheng Xu}
  received his Ph.D. degree in Statistics from the University of
  Wisconsin-Madison in 2024, focusing on distance-based methods for phylogenetic
  networks.  He currently works in a research role in the finance sector.
\end{IEEEbiographynophoto}

\vspace{-33pt}
\begin{IEEEbiographynophoto}{C{\'e}cile An{\'e}}
is currently Professor at the University of Wisconsin - Madison.
Her research interests are in the development of statistical and computational
methods for the study of molecular and trait evolution.
\end{IEEEbiographynophoto}
}

\newpage

\appendix[Example edge-based $\mu$-representation]
\renewcommand{\thefigure}{A\arabic{figure}}

\label{sec:appendix-muexample}

Let $\mathcal{L}=(a_1,a_2,b,c,d,h_1,h_2,h_3)$.
In Fig.~\ref{fig:compare2nets} left, $N$ is an $\mathcal{L}$-network.
In $\mu$-vectors, leaves are ordered as in $\mathcal{L}$.
Following notations in \Cref{alg:edge-mu-rep},
$\mue(N) = A_1 + A_2 + A_3 + A_4$ where the multisets $A_i$ ($i\leq 4$)
partition the $\mu$-vector sets
based on their type (uni/bidirectional) and tags
($\treetag$, $\hybtag$ or $\roottag$).

\noindent
$A_1$ contains the $\mu$-vector sets of tree edges in the directed part
(in black in Fig.~\ref{fig:compare2nets}),
which are here the pendent edges (incident to leaves) and $e_5$:
\begin{IEEEeqnarray*}{rcllc}
A_1 &=& \lbag
  & \{((1{,}0{,}0{,}0{,}0, 0{,}0{,}0), \treetag)\}, & \;\;\mbox{\texttt{\# pendent to }}a_1 \\
&&& \{((0{,}1{,}0{,}0{,}0, 0{,}0{,}0), \treetag)\}, &\\
&&& \{((0{,}0{,}1{,}0{,}0, 0{,}0{,}0), \treetag)\}, &\\
&&& \{((0{,}0{,}0{,}1{,}0, 0{,}0{,}0), \treetag)\}, &\\
&&& \{((0{,}0{,}0{,}0{,}1, 0{,}0{,}0), \treetag)\}, &\\
&&& \{((0{,}0{,}0{,}0{,}0, 1{,}0{,}0), \treetag)\}, &\\
&&& \{((0{,}0{,}0{,}0{,}0, 0{,}1{,}0), \treetag)\}, &\\
&&& \{((0{,}0{,}0{,}0{,}0, 0{,}0{,}1), \treetag)\}, &\\
&&& \{((0{,}0{,}0{,}0{,}0, 1{,}1{,}0), \treetag)\} \rbag\,.&\mbox{\texttt{\# }}e_5
\end{IEEEeqnarray*}

\noindent
$A_2$ contains the $\mu$-vector sets of hybrid edges
(in blue in Fig.~\ref{fig:compare2nets}):
\begin{IEEEeqnarray*}{rcllc}
A_2 &=& \lbag
  & \{((0{,}0{,}0{,}0{,}0, 1{,}1{,}1), \hybtag)\}, & \\
&&& \{((0{,}0{,}0{,}0{,}0, 1{,}1{,}1), \hybtag)\}, &\\
&&& \{((0{,}0{,}0{,}0{,}0, 1{,}1{,}1), \hybtag)\} \rbag\,.&
\end{IEEEeqnarray*}

\noindent
$A_3$ contains a single $\mu$-vector set, because $N$ has only 1 root component:
\begin{IEEEeqnarray*}{rcllc}
A_3 &=& \lbag
  & \{((1{,}1{,}1{,}1{,}1, 3{,}3{,}3), \roottag)\} \rbag\,.&
\end{IEEEeqnarray*}

\noindent
$A_4$ contains the bidirectional $\mu$-vector sets of edges in the root component
(in brown in Fig.~\ref{fig:compare2nets}):
\begin{IEEEeqnarray*}{rcllc}
A_4 &=& \lbag
  & \{((1{,}1{,}0{,}0{,}0, 0{,}0{,}0), \treetag), ((0{,}0{,}1{,}1{,}1, 3{,}3{,}3), \treetag)\},
   &\;\;\mbox{\texttt{\# }}e_1 \\
&&& \{((0{,}0{,}1{,}0{,}0, 0{,}0{,}0), \treetag), ((1{,}1{,}0{,}1{,}1, 3{,}3{,}3), \treetag)\},
   &\;\;\mbox{\texttt{\# }}e_2 \\
&&& \{((1{,}1{,}1{,}1{,}0, 1{,}1{,}1), \treetag), ((0{,}0{,}0{,}0{,}1, 2{,}2{,}2), \treetag)\},
   &\;\;\mbox{\texttt{\# }}e_3 \\
&&& \{((1{,}1{,}1{,}1{,}0, 2{,}2{,}2), \treetag), ((0{,}0{,}0{,}0{,}1, 1{,}1{,}1), \treetag)\},
   &\;\;\mbox{\texttt{\# }}e_4 \\
&&& \{((1{,}1{,}0{,}0{,}0, 1{,}1{,}1), \treetag), ((0{,}0{,}1{,}1{,}1, 2{,}2{,}2), \treetag)\},
   &\;\;\mbox{\texttt{\# }}e_6 \\
&&& \{((1{,}1{,}1{,}0{,}0, 1{,}1{,}1), \treetag), ((0{,}0{,}0{,}1{,}1, 2{,}2{,}2), \treetag)\} \rbag\,.
   &\;\;\mbox{\texttt{\# }}e_7
\end{IEEEeqnarray*}

Now we reconstruct $N$ from $\mue(N)$, following notations from \Cref{alg:canon-dag} and \Cref{alg:network-from-rooted-partner}.

\noindent
$B_1$ contains the $\mu$-vectors from those in $A_1$, corresponding to the tree
nodes in the directed part:
\begin{IEEEeqnarray*}{rcll}
B_1 &=& \lbag &
(1{,}0{,}0{,}0{,}0, 0{,}0{,}0), (0{,}0{,}0{,}1{,}0, 0{,}0{,}0), (0{,}0{,}0{,}0{,}0, 0{,}1{,}0), \\ &&&
(0{,}1{,}0{,}0{,}0, 0{,}0{,}0), (0{,}0{,}0{,}0{,}1, 0{,}0{,}0), (0{,}0{,}0{,}0{,}0, 0{,}0{,}1), \\ &&&
(0{,}0{,}1{,}0{,}0, 0{,}0{,}0), (0{,}0{,}0{,}0{,}0, 1{,}0{,}0), (0{,}0{,}0{,}0{,}0, 1{,}1{,}0) \rbag\,.
\end{IEEEeqnarray*}

\noindent
$B_2$ contains the hybrid node's $\mu$-vector (in the directed part):
\begin{IEEEeqnarray*}{rcllc}
B_2 &=& \{
  & (0{,}0{,}0{,}0{,}0, 1{,}1{,}1) \}\,.&
\end{IEEEeqnarray*}

\noindent
$B_3$ contains a single root $\mu$-vector:
\[ B_3=\{z\} \mbox{ with } z = (1{,}1{,}1{,}1{,}1, 3{,}3{,}3)\;. \]

\noindent
$M(z)$ contains all directional $\mu$-vectors of all edges in the root component
corresponding to $z$:
\begin{IEEEeqnarray*}{rcll}
M(z) &=& \lbag
  & (1{,}1{,}0{,}0{,}0, 0{,}0{,}0), (0{,}0{,}1{,}0{,}0, 0{,}0{,}0), (1{,}1{,}1{,}1{,}0, 1{,}1{,}1),\\
&&& (0{,}0{,}1{,}1{,}1, 3{,}3{,}3), (1{,}1{,}0{,}1{,}1, 3{,}3{,}3), (0{,}0{,}0{,}0{,}1, 2{,}2{,}2),\\[4pt]
&&& (1{,}1{,}1{,}1{,}0, 2{,}2{,}2), (1{,}1{,}0{,}0{,}0, 1{,}1{,}1), (1{,}1{,}1{,}0{,}0, 1{,}1{,}1),\\
&&& (0{,}0{,}0{,}0{,}1, 1{,}1{,}1), (0{,}0{,}1{,}1{,}1, 2{,}2{,}2), (0{,}0{,}0{,}1{,}1, 2{,}2{,}2) \rbag\,.
\end{IEEEeqnarray*}

\noindent
After arbitrarily choosing $r(z)=(0{,}0{,}0{,}1{,}1, 2{,}2{,}2)$,
$B_4$ contains a subset of directional $\mu$-vectors:
\begin{IEEEeqnarray*}{rcll}
B_4 &=& \lbag
  & (1{,}1{,}0{,}0{,}0, 0{,}0{,}0), (0{,}0{,}1{,}0{,}0, 0{,}0{,}0), (0{,}0{,}0{,}0{,}1, 2{,}2{,}2),\\
&&& (0{,}0{,}0{,}0{,}1, 1{,}1{,}1), (1{,}1{,}0{,}0{,}0, 1{,}1{,}1), (0{,}0{,}0{,}1{,}1, 2{,}2{,}2) \rbag\,.
\end{IEEEeqnarray*}

\begin{figure}
\centering \includegraphics{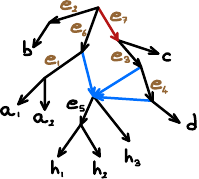}
\caption{For $N$ in Fig.~\ref{fig:compare2nets} (left),
  rooted partner $G$ from \Cref{alg:edge-mu-rep} when choosing
  $r(z)=(0{,}0{,}0{,}1{,}1, 2{,}2{,}2)$,
  which corresponds to $e_7$ (red arrow).
  The undirected edges in $N$ (brown labels) are identified by
  \Cref{alg:network-from-rooted-partner}.
}\label{fig:compare2nets-G}
\end{figure}

\noindent
Fig.~\ref{fig:compare2nets-G} shows $G$, the rooted partner
such that $\muv(G) = B_1+B_2+B_3+B_4$.
\Cref{alg:network-from-rooted-partner} starts by obtaining
$B=\mue(G)$:
\begin{IEEEeqnarray*}{rcll}
B &=& \lbag
  & \{((1{,}0{,}0{,}0{,}0, 0{,}0{,}0), \treetag)\}, \{((0{,}0{,}1{,}0{,}0, 0{,}0{,}0), \treetag)\},\\
&&& \{((0{,}1{,}0{,}0{,}0, 0{,}0{,}0), \treetag)\}, \{((0{,}0{,}0{,}1{,}0, 0{,}0{,}0), \treetag)\},\\
&&& \{((0{,}0{,}0{,}0{,}1, 0{,}0{,}0), \treetag)\}, \{((0{,}0{,}0{,}0{,}0, 0{,}1{,}0), \treetag)\},\\
&&& \{((0{,}0{,}0{,}0{,}0, 1{,}0{,}0), \treetag)\}, \{((0{,}0{,}0{,}0{,}0, 0{,}0{,}1), \treetag)\},\\
&&& \{((0{,}0{,}0{,}0{,}0, 1{,}1{,}0), \treetag)\},\\
&&& \{((0{,}0{,}0{,}0{,}0, 1{,}1{,}1), \hybtag)\}, \{((0{,}0{,}0{,}0{,}0, 1{,}1{,}1), \hybtag)\},\\
&&& \{((0{,}0{,}0{,}0{,}0, 1{,}1{,}1), \hybtag)\}, \\
&&& \{((1{,}1{,}1{,}1{,}1, 3{,}3{,}3), \roottag)\}, \\
&&& \{((1{,}1{,}0{,}0{,}0, 0{,}0{,}0), \treetag)\}, \{((0{,}0{,}1{,}0{,}0, 0{,}0{,}0), \treetag)\},\\
&&& \{((0{,}0{,}0{,}0{,}1, 2{,}2{,}2), \treetag)\}, \{((0{,}0{,}0{,}0{,}1, 1{,}1{,}1), \treetag)\},\\
&&& \{((1{,}1{,}0{,}0{,}0, 1{,}1{,}1), \treetag)\}, \{((0{,}0{,}0{,}1{,}1, 2{,}2{,}2), \treetag)\} \rbag\,.
\end{IEEEeqnarray*}

\noindent
Then $F$ contains the $\mu$-vectors $x$ of elements $\{(x,\treetag)\}$
that are in $B$ but not in $\mue(N)$:
\begin{IEEEeqnarray*}{rcll}
F &=& \lbag
  & (1{,}1{,}0{,}0{,}0, 0{,}0{,}0), (0{,}0{,}1{,}0{,}0, 0{,}0{,}0), \\
&&& (0{,}0{,}0{,}0{,}1, 2{,}2{,}2), (0{,}0{,}0{,}0{,}1, 1{,}1{,}1),\\
&&& (1{,}1{,}0{,}0{,}0, 1{,}1{,}1), (0{,}0{,}0{,}1{,}1, 2{,}2{,}2) \rbag\,.
\end{IEEEeqnarray*}
For $x = (0{,}0{,}1{,}0{,}0, 0{,}0{,}0)$, $\{(x, \treetag)\}$ has multiplicity 2
in $B$, corresponding to an elementary path of 2 edges:
$e_2$ and the pendent edge incident to $b$.
As $x$ has multiplicity 1 in $F$, only $e_2$ (in $G$) is undirected to obtain $N$.
The other elements in $F$ also have multiplicity 1.
Each one corresponds to a single edge in $G$
(labeled in brown in Fig.~\ref{fig:compare2nets-G})
and is undirected to obtain $N$.

\end{document}